\documentclass[12pt]{amsproc}

\usepackage[colorlinks,linkcolor=magenta,anchorcolor=magenta,citecolor=cyan,CJKbookmarks=True]{hyperref}
\usepackage{lineno,hyperref}
\modulolinenumbers[5]
\usepackage{amsmath}
\usepackage{amssymb}
\usepackage{mathrsfs}
\usepackage{amsthm}
\usepackage{bm}
\usepackage{graphicx}
\usepackage{booktabs}
\usepackage{float}
\usepackage{subfig}
\usepackage{geometry}
\usepackage{bbm}
\usepackage{color}
\newtheorem{Def}{Definition}[section]
\newtheorem{lem}[Def]{Lemma}
\newtheorem{tho}[Def]{Theorem}
\newtheorem{pro}[Def]{Proposition}
\newtheorem{rem}[Def]{Remark}
\newtheorem{rems}[Def]{Remarks}
\newtheorem{cor}[Def]{Corollary}
\newtheorem{hyp}[Def]{Assumption}

\newcommand{\ud}{\mathrm d}

\numberwithin{equation}{section}
\allowdisplaybreaks[4]

\begin{document}

\title[Convergence in density of numerical approximation]
{Convergence in Density of Splitting AVF Scheme for Stochastic Langevin Equation}

\author{Jianbo Cui }
\address{1. LSEC, ICMSEC, \\
Academy of Mathematics and Systems Science,\\ Chinese Academy of Sciences,\\
Beijing, 100190, China\\ 
2. School of Mathematical Science,\\ University of Chinese Academy of Sciences,\\ Beijing, 100049, China}
\curraddr{}
\email{jianbocui@lsec.cc.ac.cn}
\thanks{}

\author{Jialin Hong}
\address{1. LSEC, ICMSEC, \\
Academy of Mathematics and Systems Science,\\ Chinese Academy of Sciences,\\
Beijing, 100190, China\\ 
2. School of Mathematical Science,\\ University of Chinese Academy of Sciences,\\ Beijing, 100049, China}
\curraddr{}
\email{hjl@lsec.cc.ac.cn}
\thanks{}

\author{Derui Sheng}
\address{1. LSEC, ICMSEC, \\
Academy of Mathematics and Systems Science,\\ Chinese Academy of Sciences,\\
Beijing, 100190, China\\ 
2. School of Mathematical Science,\\ University of Chinese Academy of Sciences,\\ Beijing, 100049, China}
\curraddr{}
\email{sdr@lsec.cc.ac.cn}
\thanks{}



\date{\today}

\dedicatory{}

\begin{abstract}
In this article, we study the density function of the numerical solution of the splitting averaged vector field (AVF) scheme for the stochastic Langevin equation.
To deal with the non-globally monotone coefficient in the considered equation, we first present the exponential integrability properties of the exact and numerical solutions. 
Then we show the existence and smoothness of the density function of
 the numerical solution by proving its uniform non-degeneracy in Malliavin sense. 
 In order to analyze the approximate error between the density function of the exact solution and that of the numerical solution,
 we derive the optimal strong convergence rate in every Malliavin--Sobolev norm of the numerical scheme via Malliavin calculus. 
Combining the approximation result of Donsker's delta function and the smoothness of the density functions, we prove that the convergence rate in density  coincides with the optimal strong convergence rate of the numerical scheme.

\end{abstract}

\maketitle

\section{Introduction}

Convergence in density of numerical approximations through the probabilistic approach has received considerable attentions for stochastic differential equations (SDEs) whose coefficients are smooth vector fields with bounded derivatives. It is well known that,
under the uniform ellipticity condition, the numerical solution given by the Euler--Maruyama scheme admits a density function (see e.g. \cite{VE02}) and converges in density of order $1$ (see e.g. \cite[Theorem 8]{JG05}). Under H\"{o}rmander's condition, the idea of  perturbing the numerical solution has been used in \cite{BT96,HW96,KHA97} to approximate the density function $p_T(x,y)$ of the exact solution starting from $x$ at time $T$. 
In \cite{BT96}, the authors show that the difference between $p_T(x,y)$ and the density function of the law of a small perturbation of the Euler--Maruyama method with stepsize $\frac{T}{N}$ is expanded in terms of powers of  $\frac{1}{N}$.
 The authors in \cite{HW96} obtain a general approximation result for Donsker's delta functions and approximate 
  $p_T(x,y)$ by the density function of the sum of the It\^{o}--Taylor scheme and an independent Gaussian random variable. 
In \cite{KHA97} the author studies the It\^{o}--Taylor approximation by applying a slight modification of the weak approximation technique and proves that the rate of convergence in density can be considered as weak approximation rate.
For the numerical approximations of SDEs with superlinearly growing nonlinearities and degenerate additive noises,
to the best of our knowledge,
there are few results available concerning the convergence in density. Two natural questions are:

\textit{{\rm{(i)}} Does the density function of the numerical solution exist?}

\textit{{\rm{(ii)}} Once the density function of numerical solution exists, does it provide a proper approximation for the density function of the exact solution?}
  
To study the above questions, the present work considers the numerical approximation of the stochastic Langevin equation
\begin{equation}\label{SDE1}
\left\{
\begin{split}
&\,\ud P=-\nabla F(Q)\,\ud t-vP \,\ud t+\sigma\,\ud W_t,\\
&\,\ud Q=P\,\ud t.
\end{split}
\right.
\end{equation}
Here $t\in(0,T],\,T>0,\,v>0$, $\sigma=[\sigma_1,\ldots,\sigma_d]$ with $\sigma_k,\,k=1,\ldots,d$, being $m$-dimensional constant vectors, $-\nabla F$ is a locally Lipschitz function and   $W=(W^1,...,W^d)^\top$ is a $d$-dimensional standard Wiener process on  a filtered complete probability space $(\Omega,\mathscr{F}, \{\mathscr{F}_t\}_{t\ge0}, \mathbb{P})$. Equation \eqref{SDE1} arises in various complex dynamical system models subject to random noise such as chemical interactions and molecular dynamics, for more details, see \cite{DTG00,AB18} and references therein.  
With the help of exponential moment estimate of $X(t)=(P(t)^\top,Q(t)^\top)^\top$,
we show that $\{X(t)\}_{t\in(0,T]}$ possesses a smooth density function $\{p_t(X(0),y)\}_{t\in(0,T]}$ for equation \eqref{SDE1} under H\"ormander's condition.
In order to inherit this property in numerical approximation, we propose the splitting AVF scheme: 
\begin{equation}\label{split sol}
\left\{
\begin{split}
&\bar P_{n+1}=P_n-h\int_0^1 \nabla F\left(Q_n+\tau\left(\bar Q_{n+1}-Q_n\right)\right)\,\ud \tau,\\
&\bar Q_{n+1}=Q_n+\frac{h}{2}\left(\bar P_{n+1}+P_n\right),\\
&P_{n+1}=e^{-vh}\bar P_{n+1}+\sum_{k=1}^d \int_{t_n}^{t_{n+1}} e^{-v(t_{n+1}-t)}\sigma_k\,\ud W_{t}^{k},\\
&Q_{n+1}=\bar Q_{n+1},
\end{split}
\right.
\end{equation}
where $(P_0^\top,Q_0^\top)^\top=(P(0)^\top,Q(0)^\top)^\top$ is a deterministic datum, $h=T/N^h$ and $n=0,\ldots,N^h-1$.

With regard to the problem \textit{{\rm{(i)}}}, 
we first study the regularity estimate of the numerical solution $X_n=(P_n^\top,Q_n^\top)^\top$ in Malliavin sense.
By showing the exponential integrability property of $X_n$, we obtain its regularity estimate, in every Malliavin--Sobolev space, for equation \eqref{SDE1} with non-globally monotone coefficient.
Then 
combining this estimate with the invertibility of the corresponding Malliavin covariance matrices $\gamma_{n},\,n=2,\ldots,N^h$, we 
 prove the existence of the density functions $p^n_T(X_0,y)$ of $X_n,\,n=2,\ldots,N^h$.
Furthermore, we wonder whether $p^{N^h}_T(X_0,y)$ could 
inherit the smoothness of $p_T(X(0),y)$.
This is more involved than studying the smoothness of $p_T(X(0),y)$ due to the loss of 
H\"{o}rmander's theorem. Our solution to 
this problem lie on deriving the regularity estimate of $X_{N^h}$ and proving the non-degeneracy of $\gamma_{N^h}$.  By deducing a positive lower bound estimate of the smallest eigenvalue of $\gamma_{N^h}$, we prove that 
$\left(\det\gamma_{N^h}\right)^{-1}\in L^{\infty-}(\Omega)$. 
By means of the criterion for the smoothness of the density function of a random variable (see e.g. \cite[Theorem 2.1.4]{DN06}), we finally prove the smoothness of $p^{N^h}_T(X_0,y)$.

Concerning the problem \textit{{\rm{(ii)}}}, our strategy includes two stages. In the first stage, we derive 
 the optimal strong convergence rate of scheme \eqref{split sol} for equation \eqref{SDE1}. 
  \begin{tho}\label{SC1}
Let Assumption \ref{F2} hold, $h_0$ be a sufficiently small positive constant and $p\ge1$. There exists some positive constant $C=C(p,T,\sigma,X(0))$ such that for any $h\in(0,h_0]$,
\begin{equation*}
\sup_{n\le N^h}\left\|X_n-X(t_n)\right\|_{L^{2p}(\Omega;\mathbb{R}^{2m})}\le Ch.
\end{equation*}
\end{tho}
Up to now, there already exist a lot of strong convergence results of numerical approximations for SDEs with monotone coefficients, see  e.g. \cite{SS16,TZ13} and reference therein.
 For SDEs with non-globally monotone coefficients driven by additive noises, we are only aware that the authors in \cite{HJ14} obtain the strong convergence rate of the stopped increment-tamed Euler--Maruyama scheme. 
 To the best of our knowledge, no optimal strong convergence rate results of the numerical schemes are known for such equations.
In Theorem \ref{SC1}, we solve the problem emerged from  \cite[Remark 3.1]{HJ14} and overcome the order barrier in the strong error analysis in terms of scheme \eqref{split sol} for equation \eqref{SDE1}. The key ingredients in proving the optimal convergence rate result lie on two aspects, one being to deduce a priori strong error estimate of scheme \eqref{split sol} by the exponential integrability properties, another being the applications of the regularity estimate in Malliavin sense and Malliavin integration by parts formula.

In the second stage, we extend the strong convergence result to the convergence result in density for scheme \eqref{split sol}.
\begin{tho}\label{order}
Let Assumptions \ref{F2}-\ref{F4} hold, $\alpha>0, \beta\ge0$ and $1<p<\infty$. Then for $\alpha>\beta+2m/q+1,\, 1/p+1/q=1$, it holds that
\begin{equation*}
\sup_{y\in\mathbb{R}^{2m}}\left\|(1-\Delta)^{\beta/2}\delta_y\circ X_{N^h}-(1-\Delta)^{\beta/2}\delta_y\circ X(T)\right\|_{-\alpha,p}=\mathcal{O}(h),~as~h\rightarrow0.
\end{equation*}
\end{tho}
Here $\delta_y\circ X_{N^h}$ and $\delta_y\circ X(T)$ are Donsker's delta functions, and $\|\cdot\|_{-\alpha,p}$ denotes the norm in the Banach space $\mathbb{D}^{-\alpha,p}=(\mathbb{D}^{\alpha,q})^{\prime}$.
To the best of our knowledge, Theorem \ref{order} is the first convergence rate result in density of numerical approximations for SDEs with non-globally monotone coefficients and degenerate additive noises. 
The key ingredients in proving this convergence result are
the strong convergence analysis in every Malliavin--Sobolev norm and 
the uniform non-degeneracy property of $X_{N^h}$.
By the regularity estimates of exact and numerical solutions and Theorem \ref{SC1}, we first obtain the strong convergence in every Malliavin--Sobolev norm.  
Then combining the error estimate in Malliavin-Sobolev space $\mathbb{D}^{1,p}$ with 
$\left\|\det(\gamma_{N^h})^{-1}\right\|_{L^p(\Omega)}=\mathcal{O}\left(h^{-\nu(p)}\right)$,
we deduce the uniform non-degeneracy property of $X_{N^h}$, that is, for sufficiently small positive constant $h_0$ and for any $p\ge1$,
\begin{equation*}
 \sup_{h\in(0,h_0]}\left\|\det(\gamma_{N^h})^{-1}\right\|_{L^p(\Omega)}<\infty.
 \end{equation*}
Using the approximation result of Donsker's delta function, we finally show that the convergence rate in density  coincides with the optimal strong convergence rate for scheme \eqref{split sol}. 
We would like to mention that, the approaches to deriving the optimal strong convergence rate 
 and to deducing the convergence in density are also applicable to a number of other numerical approximations for general SDEs.

The outline of this paper is as follows. Section \ref{S2} is devoted to an introduction of Malliavin calculus, the regularity of probability laws and main assumptions on equation \eqref{SDE1}. In Section \ref{S3}, we present the exponential integrability property of the exact solution, as well as the existence and smoothness of its density function.  In Section \ref{S4}, we propose the splitting AVF scheme and show the exponential integrability property and the regularity estimate of the numerical solution in Malliavin sense. The optimal strong convergence rate of scheme \eqref{split sol} is shown in Section \ref{S5}. 
In Section \ref{S6}, we show that the numerical solution is uniformly non-degenerate and admits a smooth density function. Combined with the strong convergence in every Malliavin--Sobolev norm, we derive the optimal convergence rate of the numerical scheme in density.
Finally, several numerical experiments are presented in Section \ref{S7} to support our theoretical analysis.

\section{Preliminaries}\label{S2}
In this section, we introduce some frequently used notations and some basic elements from Malliavin calculus on the Wiener space and the regularity of probability laws, as well as main assumptions on equation \eqref{SDE1}.

Given a matrix $A\in\mathbb{R}^{m\times m}$, denote by $\lambda_i(A)$ the $i$th eigenvalue, $i=1,\cdots,m$, by
$\lambda_{min}(A)$ the smallest eigenvalue, and by
$\rho(A)$ the spectral radius of $A$. 
We use $\mathbb{H}$ to denote the Hilbert space $L^2([0,T];\mathbb{R}^d)$ endowed with the inner product $\langle g,h\rangle_{\mathbb{H}}=\int_0^T \langle g(t),h(t)\rangle_{\mathbb{R}^d}\,\ud t,\,\forall\, g,\,h\in\mathbb{H}$.  
For $\vec{l}=(l_1,\ldots,l_m)$ with $l_i\ge1,\,i=1,\ldots,m$ and $x=(x_1,\ldots,x_m)$, denote $\lfloor x\rfloor^{\vec{l}}:=\sum_{i=1}^m|x_i|^{l_i}$ and $|\,\vec{l}\,|_{\infty}:=\max_{1\le i\le m}l_i$. Throughout the paper,
we denote by $C$ a generic constant which may depend on several parameters but never on the stepsize $h$ and may change from occurrence to occurrence.

\subsection{Malliavin calculus on the Wiener space}
Some basic ingredients of Malliavin calculus are presented in this part. For further results, we refer to  \cite{YZH17,IW84,DN06}. By identifying $W(t,\omega)$ with the value $\omega(t)$ at time $t$ of an element $\omega\in C_0([0,T];\mathbb{R}^d)$, we take $\Omega=C_0([0,T];\mathbb{R}^d)$ as the Wiener space and $\mathbb{P}$ as the Wiener measure.
 For $h\in\mathbb{H}$, we set $W(h):=\sum_{k=1}^d\int_0^Th_k(t)\,\ud W^k_t$. We denote
$\mathcal{S}$ the class of smooth random variables such that $F\in\mathcal{S}$ has the form 
\begin{equation}\label{F}
F=f(W(h_1),\ldots,W(h_n)),
\end{equation}
where $f$ belongs to $C_p^\infty(\mathbb{R}^n)$, $h_i\in \mathbb{H},\, i=1,\ldots,n, \,n\ge 1$.
The derivative of a smooth random variable $F$ of the form \eqref{F} is an $\mathbb{H}$-valued random variable given by 
$DF=\sum_{i=1}^n\frac{\partial f}{\partial x_i}(W(h_1),\ldots,W(h_n))h_i.$
 For any $p\ge 1$, we denote the domain of $D$ in $L^p(\Omega)$ by $\mathbb{D}^{1,p}$, meaning that $\mathbb{D}^{1,p}$ is the closure of $\mathcal{S}$ with respect to the norm
$\|F\|_{1,p}=\left(\mathbb{E}\left[|F|^p+\|DF\|_\mathbb{H}^p\right]\right)^{\frac{1}{p}}.$
We define the iteration of the operator $D$ in such a way that for a smooth random variable $F$, the iterated derivative $D^\alpha F$ is a random variable with values in $\mathbb{H}^{\bigotimes \alpha}$.
Then for any $p\ge 1$ and integer $\alpha\ge1$, we denote by $\mathbb{D}^{\alpha,p}$ the completion of $\mathcal{S}$ with respect to the norm
\begin{equation*}
\|F\|_{\alpha,p}=\left(\mathbb{E}\left[|F|^p+\sum_{j=1}^{\alpha}\|D^jF\|_{\mathbb{H}^{\bigotimes j}}^p\right]\right)^{\frac{1}{p}}.
\end{equation*}
Define
\begin{equation}\label{LDD}
L^{\infty-}(\Omega):=\bigcap_{p\ge 1} L^p(\Omega),\qquad\mathbb{D}^{\alpha,\infty}:=\bigcap_{p\ge 1} \mathbb{D}^{\alpha,p},\qquad\mathbb{D}^\infty:=\bigcap_{p\ge 1}\bigcap_{\alpha\ge 1} \mathbb{D}^{\alpha,p}
\end{equation}
to be topological projective limits. As in the Schwartz theory of distributions, we introduce the topological dual of the Banach space $\mathbb{D}^{\alpha,p}$, by
$\mathbb{D}^{-\alpha,q}=(\mathbb{D}^{\alpha,p})^{\prime},$
where $1/p+1/q=1$, and the space of generalized Wiener functionals, by
$\mathbb{D}^{-\infty}=\bigcup_{p\ge 1}\bigcup_{\alpha\ge 1}\mathbb{D}^{-\alpha,p}.$ The natural coupling of $G\in\mathbb{D}^{\alpha,p}$ and $\Phi\in\mathbb{D}^{-\alpha,q}$ with $1/p+1/q=1$ or that of $G\in\mathbb{D}^{\infty}$ and $\Phi\in\mathbb{D}^{-\infty}$ is denoted by $\mathbb{E}[G\cdot\Phi]$.
Similarly, let $V$ be a real separable Hilbert space and we define the space $\mathbb{D}^{\alpha,p}(V)$ as the completion of $V$-valued smooth random variables with respect to the norm
\begin{equation*}
\|F\|_{\alpha,p,V}=\left(\mathbb{E}\left[\|F\|_V^p+\sum_{j=1}^{\alpha}\|D^jF\|_{\mathbb{H}^{\bigotimes j}\bigotimes V}^p\right]\right)^{\frac{1}{p}}.
\end{equation*}
When we consider $V$-valued functional, the corresponding spaces in \eqref{LDD} are denoted by $L^{\infty-}(\Omega;V)$, $\mathbb{D}^{\alpha,\infty}(V)$ and $\mathbb{D}^{\infty}(V)$, respectively.

\subsection{Regularity of probability laws}
In order to study the density function of the numerical approximation, we begin with imposing the non-degeneracy condition. 
\begin{Def}\label{Def1}
A random vector $F=(F^1,F^2,\cdots,F^m)$ whose components are in $\mathbb{D}^\infty $ is non-degenerate if the Malliavin covariance matrix 
$\gamma_F:=(\langle DF^i,DF^j\rangle_{\mathbb{H}})_{1\le i,j\le m}$
is invertible a.s. and
$(\det \gamma_F)^{-1}\in L^{\infty-}(\Omega).$
\end{Def}

It is well known that if $F$ is non-degenerate, then for every $T\in\mathcal{S}'(\mathbb{R}^m)$, $T\circ F$ can be defined in $\mathbb{D}^{-\infty}$ and 
$T\circ F\in\bigcap_{p\ge1}\bigcup_{\alpha\ge 1}\mathbb{D}^{-\alpha,p}$ (see e.g. \cite{HW96}). Here, $\mathcal{S}'(\mathbb{R}^m)$ is the space of tempered distributions.
In the particular case that $T=(1-\Delta)^{\beta/2}\delta_y,$ $\beta\ge 0,$ $y\in\mathbb{R}^m$, if $\alpha>\beta+\frac{m}{q}, 1/p+1/q=1$, then
\begin{equation}\label{delta}
T\circ F=(1-\Delta)^{\beta/2}\delta_y\circ F\in\mathbb{D}^{-\alpha,p}.
\end{equation}
$\delta_y\circ F$ is called a Donsker's delta function.  Notice that 
$\mathbb{E}[\delta_y\circ F]=\rho_F(y),$
where $\rho_F(y)$ is the density at $y$ of the probability law of $F$ (see \cite[Section 4]{IW84} for a detailed discussion). We close this part with introducing some results in \cite{HW96}, which are useful for  deriving  the convergence in density of  the numerical approximation in Section \ref{S6}.
\begin{lem}\label{pdf0}
Assume that $H_n,\,H\in \mathbb{D}^{1,\infty}(\mathbb{R}^{m})$ satisfy the following conditions:

(i) there exists $\kappa>0$ such that for any $1\le p<\infty$,
$\lim_{n\rightarrow\infty}\|H_n-H\|_{1,p,\mathbb{R}^{m}}=\mathcal{O}(n^{-\kappa}),$

(ii) $(\det\gamma_H)^{-1}\in L^{\infty-}(\Omega),$

(iii) for any $1\le p<\infty$, there exists $\nu(p)>0$ such that $\|\det(\gamma_{H_n})^{-1}\|_{L^p(\Omega)}=\mathcal{O}\left(n^{\nu(p)}\right)$ as $n\rightarrow\infty$.

Then, for any $1\le p<\infty$, we have
$\sup_n\left\|\det(\gamma_{H_n})^{-1}\right\|_{L^p(\Omega)}<\infty.$
\end{lem}
\begin{pro}\label{pdf}
Let $H_n,\, n=1,2,\cdots$ and $H$ be smooth d-dimensional Wiener functionals, i.e.,
$H_n,\, H\in \mathbb{D}^{\infty}(\mathbb{R}^{m}),$ $\alpha>0,\,\beta\ge0,\,\delta>0$ and $1<p<\infty$.
Suppose that $H_n$ and $H$ satisfy the following conditions:

(i) $H_n$ approximates $H$ in $\mathbb{D}^{\infty}(\mathbb{R}^{m})$ with order $\kappa$ $(\kappa>0)$ in the sense that for every $1\le p<\infty$ and $\alpha>0$,
$\lim_{n\rightarrow\infty}\|H_n- H\|_{\alpha,p,\mathbb{R}^{m}}=\mathcal{O}(n^{-\kappa}).$

(ii) $H$ is non-degenerate, i.e.,
$\left(\det\gamma_H\right)^{-1}\in L^{\infty-}(\Omega).$

Then for $\alpha>\beta+m/q+1, 1/p+1/q=1$, 
\begin{equation}\label{pdf11}
\sup_{y\in\mathbb{R}^{m}}\left\|\left[(1-\Delta)^{\beta/2}\phi_{n^{-\delta}}\right](H_n-y)-(1-\Delta)^{\beta/2}\delta_y\circ H\right\|_{-\alpha,p}=\mathcal{O}\left(n^{-\kappa\land\delta}\right),
\end{equation}
as $n\rightarrow\infty$, where
$\phi_\rho(x)=\left(2\pi\rho^2\right)^{-m/2}e^{-\frac{\|x\|^2}{2\rho^2}},\, x\in\mathbb{R}^{m},\,\rho>0.$
\end{pro} 
\begin{rem}\label{pdf1}
If in addition $H_n$ in Proposition \ref{pdf} is uniformly non-degenerate, i.e.\\
$\sup_n\|(\det\gamma_{H_n})^{-1}\|_{L^p(\Omega)}<\infty,$
then we have
\begin{equation*}\label{pdf2}
\sup_{y\in\mathbb{R}^{m}}\left\|(1-\Delta)^{\beta/2}\delta_y\circ H_n-(1-\Delta)^{\beta/2}\delta_y\circ H\right\|_{-\alpha,p}=\mathcal{O}(n^{-\kappa}).
\end{equation*}
\end{rem}

\subsection{Main assumptions}
In this part, we introduce main assumptions on equation \eqref{SDE1}.
To ensure the existence and uniqueness of a strong solution of equation \eqref{SDE1} (see \cite[Subsection 3.1]{HJ14}), 
we assume that $F\in C^2$ is bounded below, and 
$\limsup_{r\to 0}\sup_{y\in \mathbb R^m}\frac {\|y\|^{r}}{C_0+F(y)}<\infty.$
Here, $F$ is called bounded below if  $F(y)+C_0 > 0$ holds for any $y\in \mathbb R^m$ and some constant $C_0$.
For the purpose of getting the solvability of scheme \eqref{split sol}, we further impose the assumption that $\nabla^2 F$ is bounded below uniformly in the sense that there exists a constant $K\ge0$ such that for any $y\in \mathbb R^m$, $\lambda_{min}\left(\nabla^2 F(y)\right)\ge-K$. We remark that it is, for example, satisfied in the case that $F$ is convex. All the above assumptions are supposed to be fulfilled throughout this article. For convenience, further assumptions on the drift coefficient $F$ and the diffusion coefficient $\sigma$ that may be used in the ensuing sections are given as follows.

\begin{hyp}\label{F2}
Assume that $F\in C_p^\infty$ and there exist some constants $C_i>0,\,i=1,2,3,\,\epsilon>0$ and $\vec{l}=(l_1,\ldots,l_m)$ with integers $l_i\ge1,\,i=1,\ldots,m$, such that for any $y\in\mathbb{R}^m$, the following inequalities hold:
\begin{align}\label{FQ1}
&-C_3+C_1\lfloor y\rfloor^{2\vec{l}}\le F(y)\le C_2\lfloor y\rfloor^{2\vec{l}}+C_3,\\\label{FQ2}
&\|\nabla^2 F(y)\|\le C_2\lfloor y\rfloor^{2{\vec{l}-\epsilon\mathbbm{1}}}+C_3.
\end{align}
\end{hyp}
For simplicity, we suppose that for any multi-index $\alpha$ with $|\alpha|:=\sum_{i=1}^m\alpha_i\ge1$, it holds that 
$\|\partial^\alpha F(y)\|\le C\left(1+\|y\|^{2|\,\vec{l}\,|_{\infty}}\right)$.
The Assumption \eqref{F2} is needed to deduce the optimal strong convergence rate of the splitting AVF scheme \eqref{split sol} in Section \ref{S5}. If $F$ is a polynomial satisfying \eqref{FQ1}, then $F$ satisfies \eqref{FQ2} as well. 
It can be seen that when $|\,\vec{l}\,|_{\infty}>1$, equation \eqref{SDE1} under Assumption \ref{F2} satisfies neither globally Lipschitz condition nor globally monotone condition.
Two examples satisfying Assumption \ref{F2} are given as follows:
\begin{align*}
&(1)\,m=1,\,F(y)=\sum_{i=0}^{2\kappa}a_iy^{i},\,a_{2\kappa}>0,\,\kappa\ge1,\\
&(2)\,m=2,\,F(y)=y_1^4+y_2^6+y_1y_2+\sin y_1.
\end{align*}
For convenience, we don't consider the case that $l_i=0$ for some $i=1,\dots,m$, since all the arguments in Sections \ref{S3}-\ref{S6} still hold with a slight modification.  

\begin{hyp}\label{F4}
There are at least m vectors of $\{\sigma_1,\ldots,\sigma_d\}$ linearly independent.
\end{hyp}
It is easily verified that the noise in equation \eqref{SDE1} is degenerate and that Assumption \ref{F4} implies H\"{o}rmander's condition (see e.g \rm{\cite{HSW17}}), which indicates that the law of the exact solution $X(t)$ of equation \eqref{SDE1} is absolutely continuously with respect to the Lebesgue measure on $\mathbb{R}^{2m}$, for any $t\in(0,T]$.

\section{Stochastic Langevin equation}\label{S3}
In this section, we give the exponential integrability property and the  existence and smoothness of the density function of the exact solution for equation \eqref{SDE1}. For convenience, we rewrite \eqref{SDE1} as 
\begin{equation}\label{SDE2}
dX(t)=A_0(X(t))dt+\sum_{k=1}^d \left[\begin{array}{c}\sigma_k\\0\end{array}\right]\circ\,\ud W_{t}^{k},
\end{equation}
with
\begin{equation*}
A_0(x)=\left[\begin{array}{c}-\nabla F(Q)-vP\\P\end{array}\right],\,x=(P^\top,\,Q^\top)^\top.
\end{equation*}
\subsection{Exponential integrability property of the exact solution}
\label{EI}
Let $U(x)=K_0\left(\frac{\|P\|^2}{2}+F(Q)+C_0\right),\,x=(P^\top,\,Q^\top)^\top, \,K_0\ge1$. Then $U$ is a nonnegative functional. By applying It\^o's formula to $U(X(t))$ and a standard argument, we  show the following a priori estimate, where $X(t)=(P(t)^\top,Q(t)^\top)^\top$.
\begin{lem}\label{MB}
Let $p\ge1$, then there exists $C=C(T,\sigma,X(0),p)>0$ such that
\begin{equation}\label{ME}
\mathbb{E}\left[\sup_{0\le t\le T}\|X(t)\|^p\right]\le C.
\end{equation}

\end{lem}

Beyond the above a priori estimate of $X(t)$, the exponential integrability property is also shown, which plays a key role in the study of strong convergence rate (see e.g. \cite{CHZ19,HJ14}). Let us recall the following exponential integrability lemma (see \cite[Proposition 3.1]{CHZ16} or \cite[Corollary 2.4]{HSA}). For more applications of exponential integrability property, see the references \cite{BCH18,CH17,CHS18,HJW18} and the references therein. 
\begin{lem}\label{exp}
Let $H$ be a separable Hilbert space, $U\in\mathcal{C}^2(H;\mathbb{R})$, $\bar{U}\in L^0([0,T]\times H;\mathbb R)$, $X$ be an $H$-valued, adapted stochastic process with continuous sample paths satisfying $\int_0^T\|\mu(X_s)\|+\|\sigma(X_s)\|^2\,\ud s<\infty$ a.s., and for all $t\in[0,T]$, $X_t=X_0+\int_0^t\mu(X_s)\,\ud s+\int_0^t\sigma(X_s)\,\ud W_s$ a.s. Assume that there exists an $\mathbb{R}$-valued $\mathscr{F}_0$-measurable random variable $\beta$ such that a.s.
\begin{equation}\label{exp1}
DU(X)\mu(X)+\frac{tr[D^2U(X)\sigma(X)\sigma^*(X)]}{2}+\frac{\|\sigma^*(X)DU(X)\|^2}{2e^{\beta t}}+\bar{U}(X)\le\beta U(X),
\end{equation}
then 
\begin{equation*}
\sup_{t\in[0,T]}\mathbb{E}\left[\exp\left(\frac{U(X_t)}{e^{\beta t}}+\int_o^t\frac{\bar{U}(X_r)}{e^{\beta r}}\,\ud r\right)\right]\le\mathbb{E}\left[e^{U(X_0)}\right].
\end{equation*}
\end{lem}

 Based on Lemma \ref{exp}, the authors of \cite{HJ14} prove the exponential integrability of the exact solution of equation \eqref{SDE1} when $\sigma=\sqrt{\epsilon}I$, see the formula (4.28) in \cite[Section 4.5]{HJ14}. Here, $I$ denotes the identity matrix. We now present the exponential integrability property of the exact solution of equation \eqref{SDE1}.
\begin{pro}\label{EEI}
For any $\beta\ge K_0\left(\sum\limits_{k=1}^d \|\sigma_k\|^2-2v\right),$ there holds that
\begin{equation}\label{EqEI}
\sup_{t\in[0,T]}\mathbb{E}\left[\exp\left(\frac{U(X(t))}{e^{\beta t}}\right)\right]\le C(\beta,T)e^{U(X(0))} .
\end{equation}
\end{pro}
\begin{proof}
Take $H=\mathbb{R}^{2m}$, $\mu(x)=\left[\begin{array}{c}-\nabla F(Q)-vP\\P\end{array}\right],\sigma(x)=\left[\begin{array}{ccc}\sigma_1&\ldots&\sigma_d \\0&\ldots&0\end{array}\right],\, \bar{U}\equiv-\frac{K_0}{2}\sum\limits_{k=1}^d \|\sigma_k\|^2$ in Lemma \ref{exp}. Then a straightforward calculation, similar to the formula (4.27) in \cite[Section 4.5]{HJ14}, shows that
 \eqref{exp1} holds for any $\beta\ge K_0\left(\sum\limits_{k=1}^d \|\sigma_k\|^2-2v\right)$, and thereby \eqref{EqEI} follows from Lemma \ref{exp}.
\end{proof}

\subsection{Probability density function of the exact solution}
In this part, we show that the exact solution $X(t)$ of equation \eqref{SDE1} admits a smooth density function under Assumptions \ref{F2}-\ref{F4}, for any $t\in(0,T]$. 
By using Malliavin calculus and  the exponential integrability property,
we obtain the following result on the smoothness of the density function of $X(t)$, for any $t\in(0,T]$. 
\begin{lem}\label{NS}
Let Assumptions \ref{F2}-\ref{F4} hold. Then for any fixed $t\in(0,T]$, $X(t)$ admits an infinitely differentiable density function.
\end{lem}
\begin{proof}
Fix $t\in(0,T]$.
According to \cite[Theorem 2.1.4]{DN06} and \cite[Theorem 2.3.3]{DN06}, 
it remains to  prove that for any integer $\alpha\ge1$, $X(t)\in\mathbb{D}^{\alpha,\infty}(\mathbb R^{2m}).$ Denote $j(K):=j_{\epsilon_1},\ldots,j_{\epsilon_\eta}, r(K):=r_{\epsilon_1},\ldots,r_{\epsilon_\eta}$ with $j_{\epsilon_i}\in\{1,\ldots,d\}$ and $r_{\epsilon_i}\in[0,T],\,i\in\{1,\ldots,\eta\}$ for any subset $K=\{\epsilon_1,\cdots,\epsilon_\eta\}$ of $\{1,\ldots,\alpha\}$ with $\epsilon_1<\cdots<\epsilon_\eta$. Then by the chain rule, for $t\ge r_1 \vee\cdots\vee r_\alpha,\, i=1,\ldots,2m$,
the $\alpha$-th Malliavin derivative of $X^i(t)$ satisfies: 
\begin{align}\label{DXM1}
&D_{r_1,\ldots,r_\alpha}^{j_1,\ldots,j_\alpha}\left(X^i(t)\right)\\\nonumber
&=\int_{r_1 \vee\cdots\vee r_\alpha}^t \sum_{1\le\nu\le \alpha}\left(\partial_{k_1} \cdots\partial_{k_\nu}A_0^i\right)\left(X(s)\right)\times D_{r(I_1)}^{j(I_1)}\left[X^{k_1}(s)\right] \cdots D_{r(I_\nu)}^{j(I_\nu)}\left[X^{k_\nu}(s)\right]\,\ud s,
\end{align}
where $\sum\limits_{1\le\nu\le \alpha}$ denotes the sum over all sets of partitions $\{1,\ldots,\alpha\}=I_1\cup\cdots \cup I_{\nu},\, k_l\in\{1,\ldots,2m\},\, l=1,\ldots,\nu,$ and $\nu=1,\ldots,\alpha$, and for $t<r_1 \vee\cdots\vee r_\alpha,\, i=1,\ldots,2m,$
\begin{equation*}\label{DXM2}
D_{r_1,\ldots,r_\alpha}^{j_1,\ldots,j_\alpha}(X^i(t))=0.
\end{equation*}

Now we aim to show that for $p\ge1,\,\alpha\ge1$,
\begin{equation}\label{DK}
\sup_{r_1,\ldots,r_\alpha\in[0,T]}\mathbb E\left(\sup_{r_1 \vee\cdots\vee r_\alpha\le t\le T}\|D_{r_1,\ldots,r_\alpha}^{j_1,\ldots,j_\alpha}(X(t))\|^p\right)\le C(\alpha,p).
\end{equation}
for all choices of $j_1,\ldots,j_\alpha\in\{1,\ldots,d\}$. We prove it by an induction argument on the order $\alpha$ of the Malliavin derivative of $X(t)$.

For $\alpha=1$, the Malliavin derivative of $X(t)$ satisfies the following integral equation
\begin {align*}
D_rX(t)\textbf{1}_{\{r\le t\}}=\int_r^t (\nabla A_0)(X(s))D_rX(s)\,\ud s+\sum_{k=1}^d \left[\begin{array}{c}\sigma_k\\0\end{array}\right]\circ\,\ud W_{t}^{k}\textbf{1}_{\{r\le t\}},
\end{align*}
with $\textbf{1}_{\{r\le t\}}$ denoting the indicator function of the set $\{r\le t\}$ and
\begin{equation}\label{A0}
(\nabla A_0)(X(s))=\left[\begin{array}{cc} -vI&-\nabla^2F(Q(s))\\I&0\end{array}\right]. 
\end{equation}
By the triangle inequality and Gronwall's inequality, for any fixed $r\le t$, 
\begin{align*}
\|D_rX(t)\|&\le\|A\|\exp\left(\int_r^t \|(\nabla A_0)(X(s))\|\,\ud s\right).
\end{align*}
Due to the fact that 
\begin{align*}
\sum\limits_{i=1}^{m}\left|x_i\right|^{2l_i-\epsilon}=\sum_{i=1}^m \left|x_i\right|^{2l_i\cdot\frac{l_i-\epsilon/2}{l_i}}\le\sum_{i=1}^m \left|x_i\right|^{2l_i\cdot\frac{|\,\vec{l}\,|_{\infty}-\epsilon/2}{|\,\vec{l}\,|_{\infty}}}+C\le C(m)\left(\sum_{i=0}^m \left|x_i\right|^{2l_i}\right)^{\frac{|\,\vec{l}\,|_{\infty}-\epsilon/2}{|\,\vec{l}\,|_{\infty}}}+C,
\end{align*}
where  $x=(x_1,...,x_m)\in\mathbb{R}^m$, 
and Assumption \ref{F2}, we have
\begin{equation*}\label{Q2l}
\lfloor Q(t)\rfloor^{2\vec{l}-\epsilon\mathbbm{1}}\le C\left(\lfloor Q(t)\rfloor^{2\vec{l}}\right)^{\frac{|\,\vec{l}\,|_{\infty}-\epsilon/2}{|\,\vec{l}\,|_{\infty}}}+C\le C(U(X(t)))^{\frac{|\,\vec{l}\,|_{\infty}-\epsilon/2}{|\,\vec{l}\,|_{\infty}}}+C.
\end{equation*}
From \eqref{EqEI}, the H\"{o}lder, Young and Jensen inequalities, it follows that for   $\beta\ge K_0\sum\limits_{k=1}^d \|\sigma_k\|^2$,

\begin{align}\label{EEB}
\mathbb{E}\left[\exp\left(\int_0^T C\lfloor Q(t)\rfloor^{2{\vec{l}-\epsilon\mathbbm{1}}}\,\ud t\right)\right]
\le&\frac{1}{T} \int_0^T\mathbb{E}\left[\exp(CT\lfloor Q(t)\rfloor^{2{\vec{l}-\epsilon\mathbbm{1}}})\right] \,\ud t\\\nonumber
\le&\frac{1}{T} \int_0^T\mathbb{E}\left[\exp\left(CTU(X(t))^{\frac{|\,\vec{l}\,|_{\infty}-\epsilon/2}{|\,\vec{l}\,|_{\infty}}}+C\right)\right] \,\ud t\\\nonumber
\le&\frac{C}{T} \int_0^T\mathbb{E}\left[\exp\left(\left(\frac{U(X(t))}{e^{\beta t}}\right)^{\frac{|\,\vec{l}\,|_{\infty}-\epsilon/2}{|\,\vec{l}\,|_{\infty}}}CTe^{\beta t\frac{|\,\vec{l}\,|_{\infty}-\epsilon/2}{|\,\vec{l}\,|_{\infty}}}\right)\right] \,\ud t\\\nonumber
\le&\frac{C}{T} \int_0^T\mathbb{E}\left[\exp\left(\frac{U(X(t))}{e^{\beta t}}\right)\right]^{\frac{|\,\vec{l}\,|_{\infty}-\epsilon/2}{|\,\vec{l}\,|_{\infty}}}\,\ud t\\\nonumber
\le&\frac{C}{T} \int_0^T\mathbb{E}\left[\exp\left(\frac{U(X(t))}{e^{\beta t}}\right)\right]\,\ud t+C
\le C.
\end{align}
Since $\|(\nabla A_0)(X(s))\|\le C\lfloor Q(s)\rfloor^{2{\vec{l}-\epsilon\mathbbm{1}}}+C$, we obtain that
\begin{align}\label{M=1}
&\sup_{r\in[0,T]}\mathbb{E}\left[\sup_{t\in[r,T]}\|D_rX(t)\|^p\right]\\\nonumber
&\le\sup_{r\in[0,T]}\mathbb{E}\left[C\exp\left(\int_r^T p\left(C\lfloor Q(s)\rfloor^{2{\vec{l}-\epsilon\mathbbm{1}}}+C\right)\,\ud s\right)\right]\\\nonumber
&=\mathbb{E}\left[C\exp\left(\int_0^T p\left(C\lfloor Q(s)\rfloor^{2{\vec{l}-\epsilon\mathbbm{1}}}+C\right)\,\ud s\right)\right]\le C,
\end{align}
which completes the proof of \eqref{DK} for $\alpha=1$.

Assuming that \eqref{DK} holds up to the index $\alpha-1,\,\alpha\ge2$,  we divide the sum in \eqref{DXM1}  as
\begin{align*}
&D_{r_1,\ldots,r_\alpha}^{j_1,\ldots,j_\alpha}(X(t))\\
&=\sum_{2\le\nu\le \alpha}\int_{r_1 \vee\cdots\vee r_\alpha}^t(\partial_{k_1} \cdots\partial_{k_\nu}A_0)(X(s)) D_{r(I_1)}^{j(I_1)}\left[X^{k_1}(s)\right] \cdots D_{r(I_\nu)}^{j(I_\nu)}\left[X^{k_\nu}(s)\right]\,\ud s\\
&\quad+\sum_{\kappa=1}^{2m}\int_{r_1 \vee\cdots\vee r_\alpha}^t(\partial_\kappa A_0)(X(s))D_{r_1,\ldots,r_\alpha}^{j_1,\ldots,j_\alpha}\left(X^\kappa(s)\right)\,\ud s.
\end{align*}
By applying the triangle inequality and then taking the supremum over $t_1\le T$, we obtain
\begin{align*}
&\sup_{r_1 \vee\cdots\vee r_\alpha\le t\le t_1}\left\|D_{r_1,\ldots,r_\alpha}^{j_1,\ldots,j_\alpha}(X(t))\right\|\\
&\quad\le \sum_{2\le\nu\le \alpha}\int_{r_1 \vee\cdots\vee r_\alpha}^T\|(\partial_{k_1} \cdots\partial_{k_\nu}A_0)(X(s))\|\left\|D_{r(I_1)}^{j(I_1)}\left[X^{k_1}(s)\right]\right\| \cdots \left\|D_{r(I_\nu)}^{j(I_\nu)}\left[X^{k_\nu}(s)\right]\right\|\,\ud s\\
&\qquad+\int_{r_1 \vee\cdots\vee r_\alpha}^{t_1}\|(\nabla A_0)(X(s))\|\|D_{r_1,\ldots,r_\alpha}^{j_1,\ldots,j_\alpha}(X(s))\|\,\ud s\\
&\quad\le B(T)+\int_{r_1 \vee\cdots\vee r_\alpha}^{t_1}\|(\nabla A_0)(X(s))\|\left(\sup_{r_1 \vee\cdots\vee r_\alpha\le t\le s}\left\|D_{r_1,\ldots,r_\alpha}^{j_1,\ldots,j_\alpha}(X(t))\right\|\right)\,\ud s,
\end{align*}
where
\begin{align*}
&B(T)=\sum_{2\le\nu\le \alpha}\int_{r_1 \vee\cdots\vee r_\alpha}^T\|(\partial_{k_1} \cdots\partial_{k_\nu}A_0)(X(s))\|
\prod_{\zeta=1}^\nu\left\|D_{r(I_\zeta)}^{j(I_\zeta)}\left[X^{k_\zeta}(s)\right]\right\|\,\ud s.
\end{align*}
It follows from the Gronwall lemma that, 
\begin{align*}
\sup_{r_1 \vee\cdots\vee r_\alpha\le t\le T}\left\|D_{r_1,\ldots,r_\alpha}^{j_1,\ldots,j_\alpha}(X(t))\right\|\le B(T)\exp\left(\int_{r_1 \vee\cdots\vee r_\alpha}^T\|C(\nabla A_0)(X(s))\|\,\ud s\right).
\end{align*}
Similar to \eqref{M=1}, there holds that  
\begin{align}\label{M=M}
&\sup_{r_1,\ldots,r_\alpha\in[0,T]}\mathbb{E}\left[\exp\left({\int_{r_1 \vee\cdots\vee r_\alpha}^T\beta\|(\nabla A_0)(X(s))\|\,\ud s}\right)\right]\\\nonumber
&=\mathbb{E}\left[\exp\left({\int_0^T\beta\|(\nabla A_0)(X(s))\|\,\ud s}\right)\right]\le C,
\end{align}
for any $\beta>1$. Combining the fact that $F\in C^\infty_p$ and \eqref{A0}, for all choices of $k_i\in\{1,\ldots,2m\},\, i=1,\ldots,\nu,\,1\le\nu\le \alpha$, we deduce 
\begin{equation*}
\|(\partial_{k_1} \cdots\partial_{k_\nu}A_0)(X(s))\|\le C+\|Q(s)\|^{2|\,\vec{l}\,|_{\infty}}.
\end{equation*}
By induction assumption and the H\"{o}lder inequality, we get for any $q\ge1$,
\begin{align}\label{M2}
&\sup_{r_1,\ldots,r_\alpha\in[0,T]}\mathbb{E}\left[B(T)^q\right]\\\nonumber
&\le C\sum_{2\le\nu\le \alpha}\int_0^T\mathbb{E}\left[\|(\partial_{k_1} \cdots\partial_{k_\nu}A_0)(X(s))\|^q\prod_{\zeta=1}^\nu\left(\sup_{\substack{r_i\in[0,T]\\i\in I_\zeta}}\left\|D_{r(I_\zeta)}^{j(I_\zeta)}[X^{k_\zeta}(s)]\right\|^q\right)\right]\ud s
\le C.
\end{align}
As a result, \eqref{M=M} and \eqref{M2} implies that \eqref{DK} holds for $\alpha$ via the H\"{o}lder inequality.

It follows from \eqref{DK} that
\begin{align}\label{DKp}
\mathbb{E}&\|D^\alpha X(t)\|_{\mathbb{H}^{\bigotimes\alpha}\bigotimes\mathbb{R}^{2m}}^p=\mathbb{E}\|D^\alpha X(t)\|_{L^2\left([0,T]^\alpha;(\mathbb{R}^d)^{\bigotimes\alpha}\bigotimes\mathbb{R}^{2m}\right)}^p\\\nonumber
\le& C(T,p,\alpha)\sup_{r_1,\ldots,r_\alpha\in[0,T]}\mathbb E\left(\sup_{r_1 \vee\cdots\vee r_\alpha\le t\le T}\|D_{r_1,\ldots,r_\alpha}(X(t))\|_{(\mathbb{R}^d)^{\bigotimes\alpha}\bigotimes\mathbb{R}^{2m}}^{p}\right)
\le C,
\end{align}
which completes the proof.
\end{proof}

\section{Splitting AVF scheme} \label{S4}
The bulk of this section presents the exponential integrability property, and the existence and smoothness of the density function for the numerical solution generated through the splitting AVF scheme \eqref{split sol}. To this end, we begin with introducing the splitting AVF scheme.
Let $0=t_0<t_1<\cdots<t_{N^h-1}<t_{N^h}=T$ be a uniform partition of interval $[0,T]$, where $t_n=nh,\,n=0,\ldots,N^h$.
The main idea of constructing the splitting AVF scheme is to split equation \eqref{SDE1} as

\begin{equation*}\label{SSDE2}
\begin{split}
&\,\ud \bar P=-\nabla F(\bar Q)\,\ud t,\,\ud \bar Q=\bar P\,\ud t;\\
&\,\ud \tilde{P}=-v\tilde{P} \,\ud t+\sum_{k=1}^d \sigma_k\,\ud W_{t}^{k},\,\ud \tilde{Q}=0.
\end{split}
\end{equation*} 
Here, the first subsystem is a Hamiltonian system and the second one can be solvable exactly. For the purpose of inheriting the exponential integrability property of the exact solution $X(t)$, we discrete the first subsystem by using the AVF scheme.
Combining it with explicit expression of the exact solution of the second subsystem, we obtain the splitting AVF scheme \eqref{split sol}.
It is readily get  by \eqref{split sol} that
\begin{equation*}
Q_{n+1}=Q_n+hP_n-\frac{h^2}{2}\int_0^1\nabla F(Q_n+\tau(Q_{n+1}-Q_n))\,\ud \tau.
\end{equation*}
Define 
\begin{equation*}
Z(h,P,Q,z)=z-Q-hP+\frac{h^2}{2}\int_0^1\nabla F(Q+\tau (z-Q))\,\ud \tau,
\end{equation*}
then 
\begin{equation*}
\frac{\partial Z}{\partial z}=I+\frac{h^2}{2}\int_0^1\tau\nabla^2 F(Q+\tau (z-Q))\,\ud \tau.
\end{equation*}
Under the assumption that $\nabla^2F$ is bounded below uniformly, we have $\det\left(\frac{\partial Z}{\partial z}\right)\neq 0$  as long as $h<\frac{2}{\sqrt{K}}$, which implies that \eqref{split sol} is solvable due to the implicit function theorem. In particular, if $F$ is a convex function, the proposed scheme is solvable for any stepsize $h>0$.

\subsection{Exponential integrability property of the numerical approximation}
In this part, we prove the exponential integrability property  of $X_n$, which is helpful for deducing the strong convergence rate in Section \ref{S5}. For simplicity, we denote $\bar X_{n}:=(\bar P_{n}^\top,\bar Q_{n}^\top)^\top$ with $\bar P_{n},\bar Q_{n}$ defined by \eqref{split sol}, for $n=1,\ldots,N^h$.
\begin{pro}\label{EqE}
For any $\beta\ge K_0\left(\sum\limits_{k=1}^d \|\sigma_k\|^2-2v\right),$
\begin{equation}\label{EqNI}
\sup_{n\le N^h}\mathbb{E}\left[\exp\left(\frac{U(X_n)}{e^{\beta t_n}}\right)\right]\le C(\beta)e^{U(X(0))}.
\end{equation}
\end{pro}
\begin{proof}
Notice that the AVF scheme preserves the Hamiltonian $U$ exactly, i.e., $U(\bar X_{n+1})=U(X_n)$ for $n=0,\ldots,N^h-1$ (see e.g. \cite[Proposition 2]{CCH16}). We define an auxiliary process $\tilde X(t)=(\tilde P(t)^\top,\tilde Q(t)^\top)^\top$ satisfying
\begin{equation*}
\left\{
\begin{split}
&\ud \tilde{P}=-v\tilde{P} \,\ud t+\sum_{k=1}^d \sigma_k\,\ud W_{t}^{k},\,t\in(t_n,t_{n+1}],\\
&\ud \tilde{Q}=0
\end{split}
\right.
\end{equation*}
with $\left(\tilde P(t_n)^\top,\tilde Q(t_n)^\top\right)^\top=\left(\bar P_{n+1}^\top,\bar Q_{n+1}^\top\right)^\top,\, \forall\,n=0,\ldots,N^h-1.$
By similar arguments in the proof of \eqref{EqEI}, we obtain 
\begin{align*}
\mathbb{E}\left[\exp\left(\frac{U(\tilde X(t_{n+1}))}{e^{\beta t_{n+1}}}\right)\right]
\le\mathbb{E}\left[\exp\left(\frac{U(\tilde X(t_n))}{e^{\beta t_n}}\right)\right]\exp\left[\left(\frac{K_0}{2\beta}\sum_{k=1}^d\|\sigma_k\|^2\right)(e^{-\beta t_n}-e^{-\beta t_{n+1}})\right].
\end{align*}
Since $U(\tilde X(t_n))=U(\bar X(t_{n+1}))=U(X_n)$ and $U(\tilde X(t_{n+1}))=U(X_{n+1})$, we have
 \begin{align*}
&\mathbb{E}\left[\exp\left(\frac{U(X_{n+1})}{e^{\beta t_{n+1}}}\right)\right]\le\mathbb{E}\left[\exp\left(\frac{U(X_n)}{e^{\beta t_n}}\right)\right]\exp\left[\left(\frac{K_0}{2\beta}\sum_{k=1}^d\|\sigma_k\|^2\right)(e^{-\beta t_n}-e^{-\beta t_{n+1}})\right].
\end{align*}
As a consequence,
\begin{align*}
\sup_{n\le N^h}\mathbb{E}\left[\exp\left(\frac{U(X_n)}{e^{\beta t_n}}\right)\right]
&\le\prod_{i=0}^{N^h-1}\exp\left[\left(\frac{K_0}{2\beta}\sum_{k=1}^d\|\sigma_k\|^2\right)(e^{-\beta t_i}-e^{-\beta t_{i+1}})\right]e^{U(X(0))}\\
&\le\exp\left(\frac{K_0}{2\beta}\sum_{k=1}^d\|\sigma_k\|^2\right)e^{U(X(0))},
\end{align*}
which completes the proof.
\end{proof}
Furthermore, the following moment boundedness result of the numerical solutions $X_n$ and $\bar X_n$ is established by using It\^o's formula and  the Burkholder-Davis-Gundy inequality.
\begin{lem}\label{EE}
For any $p\ge1$,  there exists $C=C(T,\sigma,X(0),p)>0$ such that
\begin{align*}
&\mathbb{E}\left[\sup_{n\le N^h}|U(\bar X_n)|^p\right]+ \mathbb{E}\left[\sup_{n\le N^h}|U(X_n)|^p\right]\le C.
\end{align*}
\end{lem}

\subsection{Probability Density Function}
After proving the existence and smoothness of the density function of the exact solution, it's a natural question to ask whether the numerical scheme could inherit  these properties (see e.g. \cite{BT96,HW96,KHA97}). 
In particular, 
for SDEs with superlinearly growing nonlinearities and degenerate additive noises, to the best of our knowledge, there exists no result on the existence of the density function of the numerical approximation.
In this part,
we  give a probabilistic proof of the existence of the density function of the numerical solution of stochastic Langevin equation with non-globally monotone coefficient under H\"ormander's condition.
 
Compared to the continuous case, it is more involved to establish the existence of the density function of the numerical  approximation even though the H\"{o}rmander condition holds.
We would like to mention that
in general case, H\"{o}rmander's condition is not a sufficient condition for the validity of the existence of the density function of the numerical solution. 

Similar to the proof of \cite[Theorem 2.2.1]{DN06}, the Malliavin derivative of $X_{n+1}$ exists and satisfies, for $r\in[0,t_n]$,
\begin{align*}
&D_rP_{n+1}=e^{-vh}\left(D_rP_n-h\int_0^1\nabla^2 F(Q_n+\tau(Q_{n+1}-Q_n))(D_rQ_n+\tau(D_rQ_{n+1}-D_rQ_n))\ud \tau\right),\\
&D_rQ_{n+1}=D_rQ_n+hD_rP_n-\frac{h^2}{2}\int_0^1 \nabla^2 F(Q_n+\tau(Q_{n+1}-Q_n))(D_rQ_n+\tau(D_rQ_{n+1}-D_rQ_n))\,\ud \tau,
\end{align*}
and for $r\in(t_n,t_{n+1}]$,
\begin{equation}\label{DX2}
\begin{split}
&D_rP_{n+1}=e^{-v(t_{n+1}-r)}\sigma,\\
&D_rQ_{n+1}=0.
\end{split}
\end{equation}
For simplicity, we introduce the following $m\times m$ symmetric matrices,
\begin{align*}
&F_1(Q_n,Q_{n+1}):=\int_0^1\nabla^2 F(Q_n+\tau(Q_{n+1}-Q_n))\tau\,\ud \tau,\\
&F_2(Q_n,Q_{n+1}):=\int_0^1\nabla^2 F(Q_n+\tau(Q_{n+1}-Q_n))(1-\tau)\,\ud \tau,
\end{align*} 
and get
\begin{align*}
&\int_0^1\nabla^2 F(Q_n+\tau(Q_{n+1}-Q_n))(D_rQ_n+\tau(D_rQ_{n+1}-D_rQ_n))\,\ud \tau\\
&=F_1(Q_n,Q_{n+1})D_rQ_{n+1}+F_2(Q_n,Q_{n+1})D_rQ_n.
\end{align*}
Therefore, for $r\in[0,t_n]$, we have
\begin{equation*}
\left[\begin{array}{cc}I & he^{-vh}F_1(Q_n,Q_{n+1})\\0 & I+\frac{h^2}{2}F_1(Q_n,Q_{n+1})\end{array}\right]\left[\begin{array}{cc}D_rP_{n+1}\\D_rQ_{n+1}\end{array}\right]=\left[\begin{array}{cc}e^{-vh}I & he^{-vh}F_2(Q_n,Q_{n+1})\\hI & I-\frac{h^2}{2}F_2(Q_n,Q_{n+1})\end{array}\right]\left[\begin{array}{cc}D_rP_n\\D_rQ_n\end{array}\right].
\end{equation*}
Since $\nabla^2F$ is bounded below by $-K$ uniformly, we have 
\begin{align*}
&\lambda_{min}(F_1(Q_n,Q_{n+1}))=\inf_{\|y\|_2=1}\int_0^1 \tau y^\top\nabla^2F(Q_n+\tau(Q_{n+1}-Q_n))y\,\ud \tau\ge\frac{-K}{2},\\
&\lambda_{min}(F_2(Q_n,Q_{n+1}))=\inf_{\|y\|_2=1}\int_0^1(1-\tau)y^\top\nabla^2 F(Q_n+\tau(Q_{n+1}-Q_n))y\,\ud \tau\ge\frac{-K}{2},
\end{align*}
which imply that the matrix $I+\frac{h^2}{2}F_1(Q_n,Q_{n+1})$ is invertible for any $h<\frac{2}{\sqrt{K}}$ and $n=0,\ldots,N^h-1.$

 In order to judge whether $\gamma_n$ is invertible,
 we next proceed to derive a recursive relationship between $\gamma_{n+1}$ and $\gamma_n$. Notice that if $I+\frac{h^2}{2}F_1(Q_n,Q_{n+1})$ is invertible, then 
\begin{align}\label{DX}
D_rX_{n+1}=A_nD_rX_n, \,r\in[0,t_n],
\end{align}
where $D_rX_n=\left[\begin{array}{c}D_rP_n\\D_rQ_n\end{array}\right]$ and 
\begin{align*}
A_n
=\left[\begin{array}{cc}I & -he^{-vh}F_1(Q_n,Q_{n+1})\left(I+\frac{h^2}{2}F_1\left(Q_n,Q_{n+1}\right)\right)^{-1}\\0 & \left(I+\frac{h^2}{2}F_1\left(Q_n,Q_{n+1}\right)\right)^{-1}\end{array}\right]\left[\begin{array}{cc}e^{-vh}I & he^{-vh}F_2(Q_n,Q_{n+1})\\hI & I-\frac{h^2}{2}F_2(Q_n,Q_{n+1})\end{array}\right].
\end{align*}
From \eqref{DX2} and \eqref{DX}, it follows that
\begin{align}\label{MX}
\gamma_{n+1}:=&\int_0^{t_{n+1}}D_rX_{n+1}(D_rX_{n+1})^\top\,\ud r\\\nonumber
=&\int_0^{t_n}D_rX_{n+1}(D_rX_{n+1})^\top\,\ud r+\int_{t_n}^{t_{n+1}}D_rX_{n+1}(D_rX_{n+1})^\top\,\ud r\\\nonumber
=&\int_0^{t_n}A_nD_rX_n(D_rX_n)^\top A_n^\top\,\ud r+\int_{t_n}^{t_{n+1}}D_rX_{n+1}(D_rX_{n+1})^\top\,\ud r\\
=&A_n\gamma_nA_n^\top+\frac{1-e^{-2vh}}{2v}\left[\begin{array}{cc}\sigma\sigma^\top&0\\0&0\end{array}\right],\,a.s.\nonumber
\end{align}

Now we turn to showing the following regularity estimate of $X_n$ in Malliavin sense.

\begin{lem}\label{NDI}
Let Assumption \ref{F2} hold, then 
\begin{equation}\label{eq43:1}
X_n\in \mathbb D^\infty(\mathbb{R}^{2m}),\, n=1,\ldots,N^h.
\end{equation}
 More precisely, there exists a positive constant $h_0$ such that for any $h\in(0,h_0]$, $\alpha\ge1$ and $p\ge1$,
\begin{equation}\label{eq43:2}
\sup\limits_{r_1,\cdots,r_\alpha \in [0,T]}\mathbb{E}\left[\sup\limits_{r_1\lor \cdots \lor r_\alpha\le t_n\le T}\|D_{r_1,\ldots,r_\alpha}X_n\|^p\right]\le C,\, n=1,\ldots,N^h,
\end{equation}
holds for some positive constant $C=C(\alpha,p)$.

\end{lem}
\begin{proof}
Since
\eqref{eq43:1} follows from \eqref{eq43:2},  it suffices to prove \eqref{eq43:2}, which is shown by an induction argument.

Let $r_1\in(t_{i_1},t_{{i_1}+1}]$, for $0\le i_1\le N^h-1$. It follows from \eqref{DX} that for any $i_1<n\le N^h$,
\begin{align}\label{DEP}
D_{r_1}P_{n+1}&=\left(I+\frac{h^2}{2}F_1\left(Q_n,Q_{n+1}\right)\right)^{-1}e^{-vh}\left(I-\frac{h^2}{2}F_1\left(Q_n,Q_{n+1}\right)\right)D_{r_1}P_n\\\nonumber
&+\left(I+\frac{h^2}{2}F_1\left(Q_n,Q_{n+1}\right)\right)^{-1}he^{-vh}\left(F_2(Q_n,Q_{n+1})-F_1(Q_n,Q_{n+1})\right)D_{r_1}Q_n\\\nonumber
&+\left(I+\frac{h^2}{2}F_1\left(Q_n,Q_{n+1}\right)\right)^{-1}h^3e^{-vh}F_1(Q_n,Q_{n+1})F_2(Q_n,Q_{n+1})D_{r_1}Q_n,\\
D_{r_1}Q_{n+1}&=\left(I+\frac{h^2}{2}F_1\left(Q_n,Q_{n+1}\right)\right)^{-1}\left(hD_{r_1}P_n+\left(I-\frac{h^2}{2}F_2(Q_n,Q_{n+1})\right)D_{r_1}Q_n\right).\label{DEQ}
\end{align}
By the spectral mapping theorem and the symmetry of $F_1(Q_n,Q_{n+1})$, for any $n=0,\ldots,N^h-1$, we get 

\begin{align}\label{term1}
&\left\|\left(I+\frac{h^2}{2}F_1(Q_n,Q_{n+1})\right)^{-1}\right\|
=\max_{1\le i\le m}\left|\frac{1}{1+\frac{h^2}{2}\lambda_i(F_1(Q_n,Q_{n+1}))}
\right|,\\\label{term2}
&\left\|\left(I+\frac{h^2}{2}F_1(Q_n,Q_{n+1})\right)^{-1}\left(I-\frac{h^2}{2}F_1(Q_n,Q_{n+1})\right)\right\|
=\max_{1\le i\le m}\left|\frac{1-\frac{h^2}{2}\lambda_i(F_1(Q_n,Q_{n+1}))}{1+\frac{h^2}{2}\lambda_i(F_1(Q_n,Q_{n+1}))}\right|,\\\label{term3}
&\left\|\left(I+\frac{h^2}{2}F_1(Q_n,Q_{n+1})\right)^{-1}\frac{h^2}{2}F_1(Q_n,Q_{n+1})\right\|
=\max_{1\le i\le m}\left|\frac{\frac{h^2}{2}\lambda_i(F_1(Q_n,Q_{n+1}))}{1+\frac{h^2}{2}\lambda_i(F_1(Q_n,Q_{n+1}))}\right|.
\end{align}
Next we estimate these three terms separately.  Choosing $h_0\le\sqrt{\frac{2}{K}}$, combined with the fact $\lambda_{min}(F_1(Q_n,Q_{n+1}))\ge\frac{-K}{2}$, it follows that 
$1+\frac{h^2}{2}\lambda_i(F_1(Q_n,Q_{n+1}))
\ge\frac{1}{2},\,\forall\,i=1,\ldots,m.$
Therefore 
\begin{equation}\label{DEF1}
\left\|\left(I+\frac{h^2}{2}F_1(Q_n,Q_{n+1})\right)^{-1}\right\|\le2.
\end{equation}
Notice that  if $\lambda_i(F_1(Q_n,Q_{n+1}))\ge0$, the left hands of \eqref{term2} and \eqref{term3} are dominated by $1$. If 
$\frac{-K}{2}\le\lambda_i(F_1(Q_n,Q_{n+1}))<0$,
the left hands of \eqref{term2} and \eqref{term3} are bounded as 
\begin{align*}
\left\|\left(I+\frac{h^2}{2}F_1(Q_n,Q_{n+1})\right)^{-1}\left(I-\frac{h^2}{2}F_1(Q_n,Q_{n+1})\right)\right\|&\le
\frac{1+\frac{h^2K}{4}}{1-\frac{h^2K}{4}}=1+\frac{\frac{h^2K}{2}}{1-\frac{h^2K}{4}}\le1+h^2K,
\end{align*}
and 
\begin{equation*}
\left\|\left(I+\frac{h^2}{2}F_1(Q_n,Q_{n+1})\right)^{-1}\frac{h^2}{2}F_1(Q_n,Q_{n+1})\right\|
\le
\frac{\frac{h^2K}{4}}{1-\frac{h^2K}{4}}\le\frac{h^2}{2}K\le1.
\end{equation*}
Hence 
\begin{align}\label{DEF2}
&\left\|\left(I+\frac{h^2}{2}F_1(Q_n,Q_{n+1})\right)^{-1}\left(I-\frac{h^2}{2}F_1(Q_n,Q_{n+1})\right)\right\|\le1+h^2K,
\end{align}
and
\begin{align}
\label{DEF3}
&\left\|\left(I+\frac{h^2}{2}F_1(Q_n,Q_{n+1})\right)^{-1}\frac{h^2}{2}F_1(Q_n,Q_{n+1})\right\|\le1.
\end{align}
Furthermore,  \eqref{DEF3} leads to
\begin{align*}
&\left\|\left(I+\frac{h^2}{2}F_1\left(Q_n,Q_{n+1}\right)\right)^{-1}h^3e^{-vh}F_1(Q_n,Q_{n+1})F_2(Q_n,Q_{n+1})D_{r_1}Q_n\right\|\\\nonumber
&\le 2h\left\|\left(I+\frac{h^2}{2}F_1\left(Q_n,Q_{n+1}\right)\right)^{-1}\frac{h^2}{2}F_1(Q_n,Q_{n+1})\right\|\left\|F_2(Q_n,Q_{n+1})\right\|\left\|D_{r_1}Q_n\right\|\\\label{DE2}
&\le 2h\left\|F_2(Q_n,Q_{n+1})\right\|\left\|D_{r_1}Q_n\right\|.
\end{align*}
From \eqref{DEP}-\eqref{DEF3}, it follows that there exists some constant $C=C(K)$ such that
\begin {align*}
&\|D_{r_1}P_{n+1}\|\le(1+Ch^2)\|D_{r_1}P_n\|+Ch\|F_1(Q_n,Q_{n+1})\|\|D_{r_1}Q_n\|+Ch\|F_2(Q_n,Q_{n+1})\|\|D_{r_1}Q_n\|,\\
&\|D_{r_1}Q_{n+1}\|\le(1+Ch^2)\|D_{r_1}Q_n\|+Ch\|D_{r_1}P_n\|.
\end{align*}
Set $e_n=\|D_{r_1}P_n\|+\|D_{r_1}Q_n\|$, then
\begin{equation}\label{DEen}
e_{n+1}\le e_n+Ch\left(1+\left\|F_1(Q_n,Q_{n+1})\right\|+\|F_2(Q_n,Q_{n+1})\|\right)e_n.
\end{equation}
Due to \eqref{DX2} and  $r_1\in(t_{i_1},t_{{i_1}+1}]$, there exists a positive constant $C=C(\sigma)$ such that
$\|D_{r_1}P_{{i_1}+1}\|+\|D_{r_1}Q_{{i_1}+1}\|\le C(\sigma).$
The discrete Gronwall lemma and \eqref{DEen} imply that
\begin{align*}
e_{n}\le C(\sigma)\exp\left(\sum_{j={i_1}}^{n-1}Ch\left(\|F_1(Q_j,Q_{j+1})\|+\|F_2(Q_j,Q_{j+1})\|+1\right)\right),\,\forall\,i_1<n\le N^h.
\end{align*}
From the H\"{o}lder, Jensen and Young inequalities and the fact  $(N^h-{i_1})h\le T$, it 
follows that 
\begin{align*}\nonumber
\mathbb{E}\left[\sup_{{i_1}<n\le N^h}e_n^p\right]&\le C(\sigma,p)\mathbb{E}\left[\exp\left(\sum_{j={i_1}}^{N^h-1}Ch(\|F_1(Q_j,Q_{j+1})\|+\|F_2(Q_j,Q_{j+1})\|+1)\right)\right]\\\nonumber
&\le C(\sigma,p)\frac{1}{N^h-{i_1}}\sum_{j={i_1}}^{N^h-1}\mathbb{E}\Big[\exp\left(CT(\|F_1(Q_j,Q_{j+1})\|+\|F_2(Q_j,Q_{j+1})\|+1)\right)\Big].
\end{align*}
By Assumption \ref{F2} and the definitions of $F_i,\,i=1,2$, we arrive at
\begin{equation*}
\|F_i(Q_j,Q_{j+1})\|\le C(1+\lfloor Q_j\rfloor^{2\vec{l}-\epsilon\mathbbm{1}}+\lfloor Q_{j+1}\rfloor^{2\vec{l}-\epsilon\mathbbm{1}}),\, i=1,2.
\end{equation*} 
Applying the H\"{o}lder inequality and \eqref{EqNI}, for any $i_1\le j\le N^h-1$, we obtain
\begin{align}\label{DEFE}
&\mathbb{E}\left[\exp\left(CT(\|F_1(Q_j,Q_{j+1})\|+\|F_2(Q_j,Q_{j+1})\|+1\right)\right]\\\nonumber
&\le C\sup_{i_1\le j\le N^h-1}\mathbb{E}\left[\exp\left(C(\lfloor Q_j\rfloor^{2\vec{l}-\epsilon\mathbbm{1}}+\lfloor Q_{j+1}\rfloor^{2\vec{l}-\epsilon\mathbbm{1}})\right)\right]\\\nonumber\nonumber
&\le C\sup_{i_1\le j\le N^h}\mathbb{E}\left[\exp\left(C\lfloor Q_j\rfloor^{2\vec{l}-\epsilon\mathbbm{1}}\right)\right]+C\le C.
\end{align}
The above estimates, combined with the fact 
$\|D_{r_1}X_n\|^p\le C\left(p,m,d\right)e_n^p,$
yield
\begin{equation}\label{D1}
\sup\limits_{r_1\in [0,T]}\mathbb{E}\left[\sup\limits_{r_1\le t_n\le T}\|D_{r_1}X_n\|^p\right]\le C,
\end{equation}
which proves the assertion for $\alpha=1$.

\textit{Step 2}: Let $r_2\in(t_{i_2},t_{{i_2}+1}]$ for $0\le i_2\le N^h-1$. Taking the Malliavin derivatives on both sides of \eqref{DEP} and \eqref{DEQ} yields that, for any $i_1\lor i_2<n\le N^h-1$,
\begin{align}\label{DEP2}
&D_{r_2}D_{r_1}P_{n+1}\\\nonumber
&=\left(I+\frac{h^2}{2}F_1\left(Q_n,Q_{n+1}\right)\right)^{-1}e^{-vh}\left(I-\frac{h^2}{2}F_1\left(Q_n,Q_{n+1}\right)\right)D_{r_2}D_{r_1}P_n\\\nonumber
&\quad+\left(I+\frac{h^2}{2}F_1\left(Q_n,Q_{n+1}\right)\right)^{-1}he^{-vh}\left(F_2(Q_n,Q_{n+1})-F_1(Q_n,Q_{n+1})\right)D_{r_2}D_{r_1}Q_n\\\nonumber
&\quad+\left(I+\frac{h^2}{2}F_1\left(Q_n,Q_{n+1}\right)\right)^{-1}h^3e^{-vh}F_1(Q_n,Q_{n+1})F_2(Q_n,Q_{n+1})D_{r_2}D_{r_1}Q_n\\\nonumber
&\quad+D_{r_2}\left[\left(I+\frac{h^2}{2}F_1\left(Q_n,Q_{n+1}\right)\right)^{-1}\right]e^{-vh}\left(I-\frac{h^2}{2}F_1\left(Q_n,Q_{n+1}\right)\right)D_{r_1}P_n\\\nonumber
&\quad+\left(I+\frac{h^2}{2}F_1\left(Q_n,Q_{n+1}\right)\right)^{-1}e^{-vh}D_{r_2}\left[I-\frac{h^2}{2}F_1(Q_n,Q_{n+1})\right]D_{r_1}P_n\\\nonumber
&\quad+D_{r_2}\left[\left(I+\frac{h^2}{2}F_1\left(Q_n,Q_{n+1}\right)\right)^{-1}\right]he^{-vh}\left(F_2(Q_n,Q_{n+1})-F_1(Q_n,Q_{n+1})\right)D_{r_1}Q_n\\\nonumber
&\quad+\left(I+\frac{h^2}{2}F_1\left(Q_n,Q_{n+1}\right)\right)^{-1}he^{-vh}D_{r_2}\left[F_2(Q_n,Q_{n+1})-F_1(Q_n,Q_{n+1})\right]D_{r_1}Q_n\\\nonumber
&\quad+D_{r_2}\left[\left(I+\frac{h^2}{2}F_1\left(Q_n,Q_{n+1}\right)\right)^{-1}\right]h^3e^{-vh}F_1(Q_n,Q_{n+1})F_2(Q_n,Q_{n+1})D_{r_1}Q_n\\\nonumber
&\quad+\left(I+\frac{h^2}{2}F_1\left(Q_n,Q_{n+1}\right)\right)^{-1}h^3e^{-vh}D_{r_2}\left[F_1(Q_n,Q_{n+1})\right]F_2(Q_n,Q_{n+1})D_{r_1}Q_n\\\nonumber
&\quad+\left(I+\frac{h^2}{2}F_1\left(Q_n,Q_{n+1}\right)\right)^{-1}h^3e^{-vh}F_1(Q_n,Q_{n+1})D_{r_2}\left[F_2(Q_n,Q_{n+1})\right]D_{r_1}Q_n\\\nonumber
&=:J_{1n}^1+J_{2n}^1+J_{3n}^1+J_{4n}^1+J_{5n}^1+J_{6n}^1+J_{7n}^1+J_{8n}^1+J_{9n}^1+J_{10n}^1,
\end{align}
and
\begin{align}\label{DEQ2}
&D_{r_2}D_{r_1}Q_{n+1}\\\nonumber
&=\left(I+\frac{h^2}{2}F_1\left(Q_n,Q_{n+1}\right)\right)^{-1}\left[hD_{r_2}D_{r_1}P_n+(I-\frac{h^2}{2}F_2(Q_n,Q_{n+1}))D_{r_2}D_{r_1}Q_n\right]\\\nonumber
&\quad+hD_{r_2}\left[\left(I+\frac{h^2}{2}F_1\left(Q_n,Q_{n+1}\right)\right)^{-1}\right]D_{r_1}P_n\\\nonumber
&\quad+D_{r_2}\left[\left(I+\frac{h^2}{2}F_1\left(Q_n,Q_{n+1}\right)\right)^{-1}\right]\left(I-\frac{h^2}{2}F_2(Q_n,Q_{n+1})\right)D_{r_1}Q_n\\\nonumber
&\quad+\left(I+\frac{h^2}{2}F_1\left(Q_n,Q_{n+1}\right)\right)^{-1}D_{r_2}\left[I-\frac{h^2}{2}F_2(Q_n,Q_{n+1})\right]D_{r_1}Q_n\\\nonumber
&=:J_{1n}^2+J_{2n}^2+J_{3n}^2+J_{4n}^2.
\end{align}

We now claim that 
for $\iota=1,\, \kappa=4,\ldots,10$ and $ \iota=2,\,\kappa=2,3,4$,  it holds that 
\begin{equation}\label{JKI}
\mathbb{E}[\|h^{-1}J^\iota_{\kappa n}\|^q]\le C(q),
\end{equation}
 for any $q\in[1,\infty)$, $i_1\lor i_2<n\le N^h-1$.
In fact, by the chain rule, we have 
\begin{align*}
&D_{r_2}\left[\left(I+\frac{h^2}{2}F_1\left(Q_n,Q_{n+1}\right)\right)^{-1}\right]\\
&=-\frac{h^2}{2}\left(I+\frac{h^2}{2}F_1\left(Q_n,Q_{n+1}\right)\right)^{-1}D_{r_2}\left[F_1(Q_n,Q_{n+1})\right]\left(I+\frac{h^2}{2}F_1\left(Q_n,Q_{n+1}\right)\right)^{-1},\\
&D_{r_2}\left[I-\frac{h^2}{2}F_1(Q_n,Q_{n+1})\right]=-\frac{h^2}{2}D_{r_2}\left[F_1(Q_n,Q_{n+1})\right],\\
&D_{r_2}\left[I-\frac{h^2}{2}F_2(Q_n,Q_{n+1})\right]=-\frac{h^2}{2}D_{r_2}\left[F_2(Q_n,Q_{n+1})\right].
\end{align*}
From \eqref{DEF1}, the following estimation
\begin{align*}
\left\|I-\frac{h^2}{2}F_2(Q_n,Q_{n+1})\right\|\le C+\frac{h^2}{2}\left\|F_2(Q_n,Q_{n+1})\right\|,
\end{align*}
and the fact that $L^{\infty-}(\Omega)$ is an algebra, it remains to show that
\begin{align}\label{DR2F}
\|D_{r_2}F_i(Q_n,Q_{n+1})\|, \,\|F_i(Q_n,Q_{n+1})\|\in L^{\infty-}(\Omega),\, i=1,2.
\end{align} 
Combining 
\begin{eqnarray*}
D_{r_2}F_i(Q_n,Q_{n+1})=\nabla F_i(Q_n,Q_{n+1})^\top\left[\begin{array}{c}D_{r_2}Q_n \\D_{r_2}Q_{n+1}\end{array}\right], \,i=1,2,
\end{eqnarray*}
\eqref{D1} and Lemma \ref{EE}, we get \eqref{DR2F}, 
which implies that \eqref{JKI} holds.

Define $\mathrm{E}_{n}:=\|D_{r_2}D_{r_1}P_n\|+\|D_{r_2}D_{r_1}Q_n\|$. From \eqref{DEP2}, \eqref{DEQ2} and  \eqref{DEF1}-\eqref{DEF3}, it follows that
\begin{equation*}\label{DR2e}
\mathrm{E}_{n+1}\le \mathrm{E}_n+Ch(1+\|F_1(Q_n,Q_{n+1})\|+\|F_2(Q_n,Q_{n+1})\|)\mathrm{E}_n+hJ_n,
\end{equation*}
with $J_n=h^{-1}\sum \|J_{\kappa n}^\iota\|=\sum \|h^{-1}J_{\kappa n}^\iota\|$, where the sums are extended to the set $\{\iota=1,\, \kappa=4,\ldots,10;\, \iota=2,\, \kappa=2,3,4\}$. It follows from \eqref{JKI} that
\begin{equation}\label{JKI1}
\mathbb{E}[\|J_n\|^q]\le C(q),\,\forall\,i_1\lor i_2<n\le N^h-1.
\end{equation}
According to
$r_1\in(t_{i_1},t_{{i_1}+1}],\, r_2\in(t_{i_2},t_{{i_2}+1}]$, as well as \eqref{DX2}, we have
$\mathrm{E}_{(i_1\lor i_2)+1}=0$.
Since $D_{r_1,r_2}$ is a symmetric operator with respect to $r_1,\,r_2$, without loss of generality, we suppose that $i_1\le i_2$.
By using the discrete Gronwall lemma and then taking $p$th power on both sides, we obtain that for any $i_2<n\le N^h-1$,
\begin{align*}
&\mathrm{E}_n^p\le\exp\left(\sum_{j=i_2+1}^{n-1}Ch(\|F_1(Q_j,Q_{j+1})\|+\|F_2(Q_j,Q_{j+1})\|+1)\right)\left(\sum_{j=i_2+1}^{n-1}hJ_j\right)^p\\
&\le\exp\left(\sum_{j=i_2+1}^{n-1}Ch(\|F_1(Q_j,Q_{j+1})\|+\|F_2(Q_j,Q_{j+1})\|+1)\right)h^p(n-1-i_2)^{p-1}\left(\sum_{j=i_2+1}^{n-1}J_j^p\right).
\end{align*}
Subsequent proof is based on \eqref{DEFE} and \eqref{JKI1}. For $\alpha\ge3$, the desired result is achieved by a recursive argument. 
\end{proof}
\begin{rems}\label{Fk0}
Let $F\in C_p^k$ for some $k\ge2$, $t\in(0,T]$ and $n=1,\ldots,N^h$. From the proofs of Lemmas \ref{NS} and \ref{NDI}, for any $\alpha\le k-2$ and $p\ge1$, we have $X(t),\,X_n\in\mathbb{D}^{\alpha,p}(\mathbb{R}^{2m})$.

\end{rems}
Based on Lemma \ref{NDI}, we are now in a position to prove the existence of the density function of $X_n$, $n=2,\ldots,N^h$. We remark that $X_1$ is degenerate in Malliavin sense since $\gamma_1$ is not invertible.

\begin{tho}\label{MRL}
Let Assumptions \ref{F2}-\ref{F4} hold. Then for any $n\in\{2,\ldots,N^h\}$, the law of $X_n$ is absolutely continuous with respect to the Lebesgue measure on $\mathbb R^{2m}$.
\end{tho}
\begin{proof}
In view of \cite[Theorem 2.1.2]{DN06} and Lemma \ref{NDI}, it remains to prove that for $n=2,\ldots,N^h$, the Malliavin covariance matrix $\gamma_{n}$ of $X_n$, is invertible a.s. Since $\gamma_{n}=\int_0^{t_n}D_rX_n(D_rX_n)^\top\,\ud r$ is a nonnegative definite matrix, it suffices to show that $\lambda_{min}(\gamma_{n+1})>0$, a.s. $\forall\,n=1,\ldots,N^h-1$. Notice that the symmetry of $\gamma_{n}$ yields that
\begin{equation*}\label{Gam}
\lambda_{min}(\gamma_{n+1})=\min_{\substack{y=(y_1^\top,y_2^\top)^\top\in\mathbb{R}^{2m}\\\|y\|=1}}y^\top\gamma_{n+1}y.
\end{equation*} 
Since $\sigma\sigma^\top$ is invertible, we have $\frac{1-e^{-2vh}}{2v}\|y_1^\top\sigma\|^2>0$ as long as $y_1\neq0$. It suffices to show that for $y=(y_1^\top,y_2^\top)^\top$ with $\|y_2\|=1$, it holds that $y^\top\gamma_{n+1}y>0$, a.s. Now we prove $\lambda_{min}(\gamma_{n+1})>0$ by induction on $n$. 

\textit{Step 1:}
Let $n=1$. By \eqref{MX}, we have
\begin{equation*}
\begin{split}
y^\top\gamma_2y
=&\frac{1-e^{-2vh}}{2v}\left\|{y_1^\top e^{-vh}\left(I-\frac{h^2}{2}F_1(Q_1,Q_2)\right)\left(I+\frac{h^2}{2}F_1(Q_1,Q_2)\right)^{-1}\sigma}\right.\\
&\left.{\qquad\qquad\quad +hy_2^\top\left(I+\frac{h^2}{2}F_1(Q_1,Q_2)\right)^{-1}\sigma}\right\|^2+\frac{1-e^{-2vh}}{2v}\|y_1^\top\sigma\|^2.
\end{split}
\end{equation*}
Substituting $y_1=0, \|y_2\|=1$ into the above equation and using the invertibility of $\sigma\sigma^\top$ and $I+\frac{h^2}{2}F_1(Q_1,Q_2)$ lead to
\begin{align*}
y^\top\gamma_2y=\frac{1-e^{-2vh}}{2v}\left\|hy_2^\top\left(I+\frac{h^2}{2}F_1(Q_1,Q_2)\right)^{-1}\sigma\right\|^2>0.
\end{align*}

\textit{Step 2}:
Assume that $\lambda_{min}(\gamma_{n+1})>0$ holds for $n-1$.
Substituting $y_1=0, \|y_2\|=1$ and \eqref{MX} into the expression of $y^\top\gamma_{n+1}y$ gives
\begin{align*}
y^\top\gamma_{n+1}y=\left[\begin{array}{c}y_1^\top,y_2^\top\end{array}\right]A_n\gamma_n^\top A_n^\top\left[\begin{array}{c}y_1\\y_2\end{array}\right]+\frac{1-e^{-2vh}}{2v}\|y_1^\top\sigma\|^2=\left[\begin{array}{c}z_1^\top,z_2^\top\end{array}\right]\gamma_n^\top\left[\begin{array}{c}z_1\\z_2\end{array}\right],
\end{align*}
where
\begin{align*}
&z_1=hy_2^\top\left(I+\frac{h^2}{2}F_1\left(Q_n,Q_{n+1}\right)\right)^{-1},\\
&z_2=y_2^\top\left(I+\frac{h^2}{2}F_1\left(Q_n,Q_{n+1}\right)\right)^{-1}\left(I-\frac{h^2}{2}F_2(Q_n,Q_{n+1})\right).
\end{align*}
Then the desired result $y^\top\gamma_{n+1}y>0,$ a.s. follows from $z_1\neq0,$ and the induction assumption that $\gamma_n$ is invertible a.s., which completes the proof.
\end{proof}

\section{Strong convergence}\label{S5}
In this section, we present the optimal strong convergence rate of the splitting AVF scheme \eqref{split sol} under Assumption \ref{F2}. Before that, we recall the mild form of the exact solution of equation \eqref{SDE1}, for any $0\le s<t\le T,$
\begin{equation}\label{Exact sol}
\left\{
\begin{split}
&P(t)=e^{-v(t-s)}P(s)-\int_s^te^{-v(t-u)}\nabla F(Q(u))\,\ud u+\sum_{k=1}^d\int_s^te^{-v(t-u)}\sigma_k\,\ud W_{u}^{k},\\
&Q(t)=Q(s)+\int_s^t P(u)\,\ud u.
\end{split}
\right.
\end{equation}

According to the exponential integrability properties of both exact and numerical solutions, a priori  strong error estimate between $X(t_n)$ and $X_n$ is established in the following Lemma. 

\begin{lem}\label{lem2}
Let Assumption \ref{F2} hold, $h_0$ be a sufficiently small positive constant and $p\ge1$. Then there exists some positive constant $C=C(p,T,\sigma,X(0))$ such that for any $h\in(0,h_0]$,
\begin{equation*}
\sup_{n\le N^h}\|X_n-X(t_n)\|_{L^{2p}(\Omega;\mathbb{R}^{2m})}\le Ch^{1/2}.
\end{equation*}
\end{lem}
\begin{proof}
From \eqref{split sol} and \eqref{Exact sol}, it follows that 
\begin{align}\label{P}
P_{n+1}-P(t_{n+1})=&e^{-vh}(P_n-P(t_n))+\int_{t_n}^{t_{n+1}}[-e^{-vh}+e^{-v(t_{n+1}-t)}]\nabla F(Q(t))\,\ud t\\\nonumber
&+e^{-vh}\int_{t_n}^{t_{n+1}}\int_0^1R_1\,\ud \tau\,\ud t,\\\label{Q}
Q_{n+1}-Q(t_{n+1})=&Q_n-Q(t_n)+h(P_n-P(t_n))+R_2-\sum_{k=1}^d \int_{t_n}^{t_{n+1}}\int_{t_n}^t e^{-v(t-s)}\sigma_k\,\ud W_{s}^{k}\,\ud t,
\end{align}
where
\begin{align*}
R_1:&=\nabla F(Q(t))-\nabla F(Q_n+\tau(Q_{n+1}-Q_n)),\\
R_2:&=\int_{t_n}^{t_{n+1}}\int_{t_n}^{t}e^{-v(t-s)}\nabla F(Q(s))\,\ud s\,\ud t+\left(h-\frac{1-e^{-vh}}{v}\right)P(t_n)\\
&-\frac{h^2}{2}\int_0^1\nabla F(Q_n+\tau(Q_{n+1}-Q_n))\,\ud\tau.
\end{align*}
The mean value theorem yields that 
\begin{align*}
R_1&=\int_0^1 \nabla^2 F\left(\theta Q(t)+\left(1-\theta\right)\left(Q_n+\tau\left(Q_{n+1}-Q_n\right)\right)\right)(Q(t)-Q_n-\tau(Q_{n+1}-Q_n))\,\ud \theta\\
&=\int_0^1 \nabla^2 F\left(\theta Q(t)+\left(1-\theta\right)\left(Q_n+\tau\left(Q_{n+1}-Q_n\right)\right)\right)(Q(t_n)-Q_n)\,\ud \theta\\
+&\int_0^1 \nabla^2 F\left(\theta Q(t)+\left(1-\theta\right)\left(Q_n+\tau\left(Q_{n+1}-Q_n\right)\right)\right)\left(\int_{t_n}^{t}  P(s)\,\ud s-\frac{\tau h}{2}(P_n+\bar P_{n+1})\right)\,\ud \theta,
\end{align*}
where $\theta\in (0,1)$ depends on $Q(t)$ and $Q_n$, $Q_{n+1}$. 
The inequalities $1-e^{-vh}\le Ch$ and $e^{-vh}-1+vh\le Ch^2$, $\forall\,h\le1$ with $C$ independent of $h$ and Assumption \ref{F2} imply that for any $\theta\in(0,1)$, $\tau\in(0,1)$, $t\in[t_n,t_{n+1}]$ and $n=0,\ldots,N^h-1$, 
\begin{align*}
\|\nabla F(Q(t))\|&\le C+\|Q(t)\|^{2|\,\vec{l}\,|_{\infty}},\\
\|\nabla F(Q_n+\tau(Q_{n+1}-Q_n))\|&\le C+\|Q(t)\|^{2|\,\vec{l}\,|_{\infty}}+\|Q_{n+1}\|^{2|\,\vec{l}\,|_{\infty}},\\
\|\nabla^2 F(\theta Q(t)+(1-\theta)(Q_n+\tau(Q_{n+1}-Q_n)))\|&\le C(1+\lfloor Q_n\rfloor^{2{\vec{l}-\epsilon\mathbbm{1}}}+\lfloor Q_{n+1}\rfloor^{2{\vec{l}-\epsilon\mathbbm{1}}}+\lfloor Q(t)\rfloor^{2{\vec{l}-\epsilon\mathbbm{1}}}).
\end{align*}
Applying the Young inequality and the triangle inequality, we get
\begin{align}\label{PN}
&\|P_{n+1}-P(t_{n+1})\|\le \|P_n-P(t_n)\|+G_n\|Q(t_n)-Q_n\|+Ch^2K_{1n},\\\label{QN}
&\|Q_{n+1}-Q(t_{n+1})\|\le \|Q_n-Q(t_n)\|+h\|P(t_n)-P_n\|+Ch^2K_{2n}+\|\eta_n\|,
\end{align} 
where 
\begin{align*}
K_{1n}&=\left(1+\sup_{t\in[0,T]}\|Q(t)\|^{2|\,\vec{l}\,|_{\infty}}+\|Q_n\|^{2|\,\vec{l}\,|_{\infty}}+\|Q_{n+1}\|^{2|\,\vec{l}\,|_{\infty}}\right)\left(\sup_{t\in[0,T]}\|P(t)\|+\|P_n\|+\|\bar P_{n+1}\|+1\right)\\
K_{2n}&=1+\sup_{t\in[0,T]}\|Q(t)\|^{2|\,\vec{l}\,|_{\infty}}+\|Q_n\|^{2|\,\vec{l}\,|_{\infty}}+\|Q_{n+1}\|^{2|\,\vec{l}\,|_{\infty}}+\sup_{t\in[0,T]}\|P(t)\|,\\
\eta_n&=\sum_{k=1}^d \int_{t_n}^{t_{n+1}}\int_{t_n}^t \sigma_k\,\ud W_{s}^{k}\,\ud t,
\end{align*}
and 
\begin{equation}\label{Gn}
G_n=\int_{t_n}^{t_{n+1}} C\left(1+\lfloor Q(t)\rfloor^{2\vec{l}-\epsilon\mathbbm{1}}+\lfloor 
Q_n\rfloor^{2\vec{l}-\epsilon\mathbbm{1}}+\lfloor Q_{n+1}\rfloor^{2\vec{l}-\epsilon\mathbbm{1}}\right)\,\ud t.\\
\end{equation}
Define $\mathcal{E}_{n+1}:=\|P_{n+1}-P(t_{n+1})\|+\|Q_{n+1}-Q(t_{n+1})\|$. The estimates \eqref{PN} and \eqref{QN} lead to 
\begin{align*}\label{SE1}
&\mathcal{E}_{n+1} \le\mathcal{E}_{n}+(h+G_n)\mathcal {E}_{n}+K_n,
\end{align*}
where 
$K_n=Ch^2K_{1n}+Ch^2K_{2n}+\|\eta_n\|\le Ch^2K_{1n}+\|\eta_n\|.$
Using the discrete Gronwall lemma and $\mathcal{E}_0=0$, we obtain
\begin {equation*}
\mathcal{E}_{n+1} \le \left(\sum_{j=0}^{n} K_j\right)\exp\left(\sum_{j=0}^{n} (h+G_j)\right),\,\forall\,n=0,\ldots,N^h-1.
\end {equation*}
Taking $p$th power on both sides and applying the H\"{o}lder inequality, we have
\begin {align}\label{SE2}
\mathcal{E}_{n+1}^p &\le \left(\sum_{j=0}^{n} K_j\right)^p\exp\left(\sum_{j=0}^{n} p(h+G_j)\right)\\\nonumber
&\le n^{p-1}\left(\sum_{j=0}^{n} K_j^p\right)\exp\left(pT\right)\exp\left(\sum_{j=0}^{n} pG_j\right),\,\forall\, n=0,\ldots,N^h-1.
\end {align}
The H\"{o}lder inequality, together with Lemmas \ref{MB}, \ref{EE} implies that
\begin {equation}\label{K12}
\left\|K_{1j}^{p}\right\|_{L^2(\Omega)} \le C,\,\left\|K_{2j}^{p}\right\|_{L^2(\Omega)} \le C,\, \forall \,j=0,\ldots,N^h-1.
\end {equation}
The stochastic Fubini theorem and the H\"{o}lder inequality lead to
\begin {align}\label{eta}
\left\|\|\eta_j\|^{p}\right\|_{L^2(\Omega)}^2&=\mathbb{E}\left[\left\|\sum_{k=1}^d \int_{t_j}^{t_{j+1}}\int_{t_j}^t \sigma_k\,\ud W_{s}^{k}\,\ud t\right\|^{2p}\right]
=\mathbb{E}\left[\left\|\sum_{k=1}^d \int_{t_j}^{t_{j+1}}(t_{j+1}-s) \sigma_k\,\ud W_{s}^{k}\right\|^{2p}\right]\\\nonumber
&\le C\sum_{k=1}^d \mathbb{E}\left[\left\|\int_{t_j}^{t_{j+1}}(t_{j+1}-s) \sigma_k\,\ud W_{s}^{k}\right\|^{2p}\right]\le Ch^{3p},\,\forall\,j=0,\ldots,N^h-1.
\end {align}
Combining the above estimates together, we obtain that for $n=0,\ldots,N^h-1$,
\begin {align*}\label{KP}
\left\|\sum_{j=0}^{n} K_j^p\right\|_{L^2(\Omega)}&\le\sum_{j=0}^{n} \left\|K_j^p\right\|_{L^2(\Omega)}
\le \sum_{j=0}^{n}\left(Ch^{2p}\|K_{1j}^{p}\|_{L^2(\Omega)}+C\left\|\|\eta_j\|^{p}\right\|_{L^2(\Omega)}\right)
\le Ch^{\frac{3p}{2}-1}.
\end {align*} 
Further, \eqref{DEFE} and the Jensen inequality imply that $\mathbb{E}\left[\exp\left(\sum_{j=1}^{n+1} Ch\lfloor Q_j\rfloor^{2\vec{l}-\epsilon\mathbbm{1}}\right)\right]\le C.$ Consequently, according to the H\"{o}lder inequality, we have
\begin{eqnarray}\label{SE3}
&&\left\|\exp\left(\sum_{j=0}^{n}pG_j\right)\right\|_{L^2(\Omega)}\\\nonumber
&&\le\exp(CT)\left\|\exp\left(\int_0^T C\lfloor Q(t)\rfloor^{2\vec{l}-\epsilon\mathbbm{1}}\,\ud t\right)\right\|_{L^4(\Omega)}\left\|\exp\left(\sum_{j=1}^{n+1} 2Ch\lfloor Q_j\rfloor^{2\vec{l}-\epsilon\mathbbm{1}}\right)\right\|_{L^4(\Omega)}\\\nonumber
&&\le C.
\end{eqnarray}
From the estimates \eqref{SE2}-\eqref{SE3}, we deduce that
$\mathbb{E}\left[\mathcal{E}_{n}^p\right]\le Ch^{\frac{p}{2}},\,\forall\,n=1,\ldots,N^h,$ which together with the fact that $\|X_n-X(t_n)\|^{p}\le C\mathcal{E}_n^p$ completes the proof.
\end{proof}

With a slight modified procedure, we get the following strong convergence result.
\begin{cor}
Let Assumption \ref{F2} hold, $h_0$ be a sufficiently small positive constant and $p\ge1$. Then there exists some positive constant $C=C(X(0),p,T,\sigma)$ such that for any $h\in(0,h_0]$,
\begin{equation*}
\left\|\sup_{n\le N^h}\|X_n-X(t_n)\|\right\|_{L^{2p}(\Omega)}\le Ch^{1/2}.
\end{equation*}
\end{cor}
\begin{proof}
Taking supreme over $n\le N^h-1$ and square on both sides of \eqref{SE2} yields
\begin {align*}
\mathbb{E}\left[\sup_{n\le N^h-1}\mathcal{E}_{n+1}^{2p}\right]&\le \left(\sum_{j=0}^{N^h-1} K_j\right)^{2p}\exp\left(\sum_{j=0}^{N^h-1} 2p(h+G_j)\right)\\\nonumber
&\le (N^h)^{2p-1}\left(\sum_{j=0}^{N^h-1} K_j^{2p}\right)\exp\left(2pT\right)\exp\left(\sum_{j=0}^{N^h-1} 2pG_j\right).
\end {align*}
Similar to the proof of Lemma \ref{lem2}, we complete the proof.
\end{proof}
The optimal strong convergence order of the numerical approximation which only use  the increments of the Wiener process is known to be $1$ for SDEs with Lipschitz and regular coefficients driven by additive noises (see e.g. \cite{CC80}). However,  for SDEs with non-globally monotone coefficients driven by additive noises, it seems that there exists a order barrier to achieve optimal strong rate (see e.g. \cite{HJ14}). 
In this part, we overcome the order barrier of the proposed scheme \eqref{split sol} by using the Malliavin integration by parts formula and Lemma  \ref{lem2}. To this end, the following a priori estimate is needed to the proof of Theorem \ref{SC1}. 
\begin{lem}\label{DH}
Let Assumption \ref{F2} hold, $h_0$ be a sufficiently small positive constant and $p\ge1$.  For any positive constant $K_1$, there exists some positive constant $C=C(p,K_1)>0$ such that for any $r\in[0,T],\,k\in\{1,...,d\},$ $0\le j< n\le N^h,$ $h\in(0,h_0]$,
\begin{equation*}
\mathbb{E} \left[\left(D_r^k\left(\prod_{i=j+1}^{n}\left(1+K_1(h+G_i)\right)\right)\right)^{2p}\right]<C,
\end{equation*}
where $G_i$ is defined by \eqref{Gn}.
\end{lem}
\begin{proof}
Since  $X_n$ and $X(t)$ are differentiable in Malliavin sense, and $G_i$ is a functional of $Q(t),\,Q_i,\,Q_{i+1}$, the Malliavin derivative of $G_i$ exists (see e.g. \cite[Chapter 1]{DN06}).
By the chain rule, the H\"{o}lder inequality and the estimation \eqref{SE3}, we obtain
\begin{align}\label{DH1}
&\mathbb{E} \left[\left(D_r^k\left(\prod_{i=j+1}^{n}\left(1+K_1(h+G_i)\right)\right)\right)^{2p}\right]\\\nonumber
&=\mathbb{E}\left[\left(\sum_{i=j+1}^{n}\prod_{\substack{\kappa=j+1\\ \kappa\neq i}}^{n}\left(1+K_1(h+G_\kappa)\right)K_1D_r^kG_i\right)^{2p}\right]\\\nonumber
&\le(n-j)^{2p-1}\mathbb{E}\left[\sum_{i=j+1}^{n}\left(\prod_{\substack{\kappa=j+1\\ \kappa\neq i}}^{n}\left(1+K_1(h+G_\kappa)\right)K_1D_r^kG_i\right)^{2p}\right]\\\nonumber
&\le C(n-j)^{2p-1}\sum_{i=j+1}^{n}\mathbb{E}\left[\exp\left(2pK_1\sum_{\kappa=j+1}^{n}(h+G_\kappa)\right)\left(D_r^kG_i\right)^{2p}\right]\\\label{DH1}
&\le C(n-j)^{2p-1}\sum_{i=j+1}^{n}\left(\mathbb{E}\left[\left(D_r^kG_i\right)^{2q}\right]\right)^{\frac{p}{q}},\nonumber
\end{align}
where $q>p$. The chain rule, the H\"{o}lder inequality and the Fubini theorem yield that 
\begin{align*}
&\mathbb{E}\left[\left(D_r^kG_i\right)^{2q}\right]\\
&=\mathbb{E}\left[\left(\int_{t_i}^{t_{i+1}} CD_r^k\left(1+\lfloor Q(t)\rfloor^{2\vec{l}-\epsilon\mathbbm{1}}+\lfloor Q_i\rfloor^{2\vec{l}-\epsilon\mathbbm{1}}+\lfloor Q_{i+1}\rfloor^{2\vec{l}-\epsilon\mathbbm{1}}\right)\,\ud t\right)^{2q}\right]\\
&\le C h^{2q-1}\mathbb{E}\left[\int_{t_i}^{t_{i+1}} \left|D_r^k\left(1+\lfloor Q(t)\rfloor^{2\vec{l}-\epsilon\mathbbm{1}}+\lfloor Q_i\rfloor^{2\vec{l}-\epsilon\mathbbm{1}}+\lfloor Q_{i+1}\rfloor^{2\vec{l}-\epsilon\mathbbm{1}}\right)\right|^{2q}\,\ud t\right]\\
&= C h^{2q-1}\int_{t_i}^{t_{i+1}}\mathbb{E}\left[ \left|D_r^k\left(1+\lfloor Q(t)\rfloor^{2\vec{l}-\epsilon\mathbbm{1}}+\lfloor Q_i\rfloor^{2\vec{l}-\epsilon\mathbbm{1}}+\lfloor Q_{i+1}\rfloor^{2\vec{l}-\epsilon\mathbbm{1}}\right)\right|^{2q}\right]\,\ud t\\
&\le Ch^{2q-1}\int_{t_i}^{t_{i+1}}\mathbb{E}\left[\left|D_r^k\left(\lfloor Q(t)\rfloor^{2\vec{l}-\epsilon\mathbbm{1}}\right)\right|^{2q}+\left|D_r^k\left(\lfloor Q_i\rfloor^{2\vec{l}-\epsilon\mathbbm{1}}\right)\right|^{2q}+\left|D_r^k\left(\lfloor Q_{i+1}\rfloor^{2\vec{l}-\epsilon\mathbbm{1}}\right)\right|^{2q}\right]\,\ud t.
\end{align*}
Furthermore, for any $r,\,t\in[0,T],\,k=1,\ldots,d$, by \eqref{DK} and Lemma \ref{MB}, we have
\begin{align*}
&\mathbb{E}\left[\left|D_r^k\left(\lfloor Q(t)\rfloor^{2\vec{l}-\epsilon\mathbbm{1}}\right)\right|^{2q}\right]\\
&=\mathbb{E}\left[\left|\sum_{\beta=1}^m(2l_\beta-\epsilon)Q_\beta(t)^{2l_\beta-\epsilon-1}D_r^kQ_\beta(t)\right|^{2q}\right]\\
&\le C\sum_{\beta=1}^m\mathbb{E}\left|Q_\beta(t)^{2l_\beta-\epsilon-1}D_r^kQ_\beta(t)\right|^{2q}\\
&\le C\sum_{\beta=1}^m\left(\mathbb{E}\left[\|Q(t)\|^{4q(2|\,\vec{l}\,|_{\infty}-\epsilon-1)}\right]\right)^{\frac{1}{2}}\left(\mathbb{E}\left|D_r^kQ_\beta(t)\right|^{4q}\right)^{\frac{1}{2}}\\
&\le C.
\end{align*}
Likewise, by Lemma \ref{NDI} and  Lemma \ref{EE}, for any $r\in[0,T],\,i=1,..,N^h,\,k=1,\ldots,d$, we have
$\mathbb{E}\left[\left|D_r^k\left(\lfloor Q_i\rfloor^{2\vec{l}-\epsilon\mathbbm{1}}\right)\right|^{2q}\right]\le C.$
Combining the above estimates together, we get 
\begin{equation}\label{DH3}
\mathbb{E}\left[(D_r^kG_i)^{2q}\right]\le Ch^{2q}.
\end{equation}
Combining \eqref{DH1} and \eqref{DH3}, we complete the proof. 
\end{proof}

Based on Lemmas  \ref{lem2} and \ref{DH}, now we prove the main result of this section. 

\textit{Proof of Theorem \ref{SC1}}
We begin with establishing a refined estimate of the error between $Q(t_{n+1})$ and $Q_{n+1}$.
By \eqref{Q}, $\|R_2\|\le Ch^2K_{2n}$ and choosing $h_0\le1$, we obtain by the Young inequality that
\begin{align*}
&\|Q_{n+1}-Q(t_{n+1})\|^2\\
&=\|Q_n-Q(t_n)\|^2+h^2\|P_n-P(t_n)\|^2+\|R_2\|^2+\|\eta_n\|^2+2h(Q_n-Q(t_n))^\top(P_n-P(t_n))\\
&\quad+2(Q_n-Q(t_n))^\top R_2-2(Q_n-Q(t_n))^\top\eta_n+2h(P_n-P(t_n))^\top R_2-2h(P_n-P(t_n))^\top\eta_n\\
&\quad-2R_2^\top\eta_n\\
&\le\|Q_n-Q(t_n)\|^2+h^2\|P_n-P(t_n)\|^2+Ch^4K_{2n}^2+\|\eta_n\|^2+2h(Q_n-Q(t_n))^\top(P_n-P(t_n))
\\
&\quad+h\|Q_n-Q(t_n)\|^2+Ch^3K_{2n}^2-2(Q_n-Q(t_n))^\top\eta_n+h\|P_n-P(t_n)\|^2+Ch^5K_{2n}^2\\
&\quad+h\|P_n-P(t_n)\|^2+h\|\eta_n\|^2+\|\eta_n\|^2+Ch^4K_{2n}^2\\
&\le {\|Q_n-Q(t_n)\|^2+Ch\mathcal{E}_n^2+C\|\eta_n\|^2+Ch^3K_{2n}^2}-2(Q_n-Q(t_n))^\top\eta_n.
\end{align*}
Further,
\begin{align}\label{Q2p}
&\|Q_{n+1}-Q(t_{n+1})\|^{2p}\\\nonumber
&=\|Q_n-Q(t_n)\|^{2p}+p\|Q_n-Q(t_n)\|^{2p-2}\left({Ch\mathcal{E}_n^2+C\|\eta_n\|^2+Ch^3K_{2n}^2}\right)\\\nonumber
&-2p\|Q_n-Q(t_n)\|^{2p-2}(Q_n-Q(t_n))^\top\eta_n\\\nonumber
&+\sum_{\kappa=2}^pC\|Q_n-Q(t_n)\|^{2p-2\kappa}\left(Ch\mathcal{E}_n^2+{C\|\eta_n\|^2+Ch^3K_{2n}^2}-2(Q_n-Q(t_n))^\top\eta_n\right)^\kappa\\\nonumber
&=:\|Q_n-Q(t_n)\|^{2p}+I_1+I_2+I_3.
\end{align}
From the Young inequality and the H\"{o}lder inequality, it follows that
\begin{align*}
I_1&\le ph\|Q_n-Q\left(t_n\right)\|^{2p}+ph^{1-p}\left(Ch\mathcal{E}_n^2+C\|\eta_n\|^2+Ch^3K_{2n}^2\right)^p\\
&\le ph\|Q_n-Q\left(t_n\right)\|^{2p}+ph^{1-p}\left(Ch^p\mathcal{E}_n^{2p}+C\|\eta_n\|^{2p}+Ch^{3p}K_{2n}^{2p}\right)\\
&\le Ch\mathcal{E}_n^{2p}+Ch^{1-p}\|\eta_n\|^{2p}+Ch^{2p+1}K_{2n}^{2p}
\end{align*}
and
\begin{align*}
I_3&\le\sum_{\kappa=2}^pC\|Q_n-Q(t_n)\|^{2p-2\kappa}\left(Ch\mathcal{E}_n^2+C\|\eta_n\|^2+Ch^3K_{2n}^2\right)^\kappa\\
&\quad+\sum_{\kappa=2}^pC\|Q_n-Q(t_n)\|^{2p-2\kappa}\left(\|Q_n-Q(t_n)\|\|\eta_n\|\right)^\kappa\\
&\le C\sum_{\kappa=2}^p\|Q_n-Q(t_n)\|^{2p-2\kappa}\left(h^\kappa\mathcal{E}_n^{2\kappa}+\|\eta_n\|^{2\kappa}+h^{3\kappa}K_{2n}^{2\kappa}\right)+C\sum_{\kappa=2}^p\|Q_n-Q(t_n)\|^{2p-\kappa}\|\eta_n\|^\kappa\\
&\le Ch^2\mathcal{E}_n^{2p}+C\sum_{\kappa=2}^p\left(h\|Q_n-Q(t_n)\|^{2p}+h^{1-p/\kappa}\left(\|\eta_n\|^{2\kappa}+h^{3\kappa}K_{2n}^{2\kappa}\right)^{\frac{p}{\kappa}}\right)\\
&\quad+\sum_{\kappa=2}^p\left(h\|Q_n-Q(t_n)\|^{2p}+h^{1-p}\|\eta_n\|^{2p}\right)\\
&\le Ch\mathcal{E}_n^{2p}+Ch^{1-p}\|\eta_n\|^{2p}+Ch^{\frac{5}{2}p+1}K_{2n}^{2p}.
\end{align*}
Substituting the above two inequalities into \eqref{Q2p} gives
\begin{align}\label{Q2p1}
\|Q_{n+1}-Q(t_{n+1})\|^{2p}&\le\|Q_n-Q(t_n)\|^{2p}+Ch\mathcal{E}_n^{2p}+Ch^{2p+1}K_{2n}^{2p}+Ch^{1-p}\|\eta_n\|^{2p}\\\nonumber
&\quad-C\|Q_n-Q(t_n)\|^{2p-2}(Q_n-Q(t_n))^\top\eta_n,\nonumber
\end{align}
where we used $h\le1$. Now we turn to estimating $P_{n+1}-P(t_{n+1})$. Taking $2p$th power on both sides of \eqref{PN}, we get
\begin{align*}
\|P_{n+1}-P(t_{n+1})\|^{2p}&\le\left(\|P_n-P(t_n)\|+G_n\|Q(t_n)-Q_n\|+Ch^2K_{1n}\right)^{2p}\\
&=\|P_n-P(t_n)\|^{2p}+2p\|P_n-P(t_n)\|^{2p-1}\left(G_n\|Q(t_n)-Q_n\|+Ch^2K_{1n}\right)\\
&\quad+\sum_{\kappa=2}^{2p}C\|P_n-P(t_n)\|^{2p-\kappa}\left(G_n\|Q_n-Q(t_n)\|+Ch^2K_{1n}\right)^\kappa.
\end{align*}
According to the Young inequality, for $2\le\kappa\le 2p$,
\begin{align*}
&\|P_n-P(t_n)\|^{2p-\kappa}\left(G_n\|Q_n-Q(t_n)\|+Ch^2K_{1n}\right)^\kappa\\
&=\|P_n-P(t_n)\|^{2p-\kappa}\left(G_n\|Q_n-Q(t_n)\|+Ch^2K_{1n}\right)^{\frac{2p-\kappa}{2p-1}}\\
&\quad\left(G_n\|Q_n-Q(t_n)\|+Ch^2K_{1n}\right)^{\kappa-\frac{2p-\kappa}{2p-1}}\\
&\le\frac{2p-\kappa}{2p-1}\|P_n-P(t_n)\|^{2p-1}\left(G_n\|Q(t_n)-Q_n\|+Ch^2K_{1n}\right)\\
&\quad+\frac{\kappa-1}{2p-1}\left(G_n\|Q_n-Q(t_n)\|+Ch^2K_{1n}\right)^{2p}.
\end{align*}
Combining the above two estimates and the H\"{o}lder inequality, we obtain
\begin{align}\label{P2p}
\|P_{n+1}-P(t_{n+1})\|^{2p}
&\le\|P_n-P(t_n)\|^{2p}+CG_n^{2p}\|Q_n-Q(t_n)\|^{2p}+Ch^{4p}K_{1n}^{2p}\\\nonumber
&\quad+C\|P_n-P(t_n)\|^{2p-1}\left(G_n\|Q_n-Q(t_n)\|+Ch^2K_{1n}\right)\\\nonumber
&\le\|P_n-P(t_n)\|^{2p}+C(h+G_n)\mathcal{E}_n^{2p}+Ch^{1+2p}K_{1n}^{2p}+CG_n^{2p}\mathcal{E}_n^{2p}.
\end{align}
Define $\mathcal{S}_{n+1}:=\left(\|P_{n+1}-P(t_{n+1})\|^{2p}+\|Q_{n+1}-Q(t_{n+1})\|^{2p}\right)^{\frac{1}{2p}}$. Note that $\mathcal{E}_n^{2p}\le C\mathcal{S}_n^{2p}$. Then it follows from \eqref{Q2p1} and \eqref{P2p} that 
$\mathcal{S}_{n+1}^{2p}\le\mathcal{S}_n^{2p}+C(h+G_n)\mathcal{S}_n^{2p}+T_n,$
where $T_n=T_{1n}+T_{2n}$ with
\begin{align*}
&T_{1n}=Ch^{2p+1}K_{1n}^{2p}+Ch^{1-p}\|\eta_n\|^{2p}+Ch^{2p+1}K_{2n}^{2p}+CG_n^{2p}\mathcal{S}_n^{2p},\\
&T_{2n}=C\|Q_n-Q(t_n)\|^{2p-2}(Q(t_n)-Q_n)^\top\eta_n.
\end{align*} 
Notice that $\mathcal{S}_{0}=0$.
The discrete Gronwall lemma (see e.g. \cite[Lemma 1.4.2]{QV94}) yields that 
\begin{align}\label{En+1}
\mathcal{S}_{n+1}^{2p}&\le\sum_{j=0}^{n}\left(\prod_{i=j+1}^{n}(1+C(h+G_i))\right)T_{1j}+\sum_{j=0}^{n}\left(\prod_{i=j+1}^{n}(1+C(h+G_i))\right)T_{2j}\\\nonumber
&\le\sum_{j=0}^{n}\exp\left({\sum\limits_{i=j+1}^{n}C(h+G_i)}\right)T_{1j}+\sum_{j=0}^{n}\left(\prod_{i=j+1}^{n}\left(1+C(h+G_i)\right)\right)T_{2j},
\end{align}
with the conventions $\prod_{i=n+1}^{n}(1+C(h+G_i))=1$ and $\sum_{i=n+1}^{n}C(h+G_i)=0$.
Now, we estimate the above two sums separately. For the first summand, Lemma  \ref{lem2} and estimations \eqref{K12}-\eqref{SE3} yield that,
\begin{align*}
\mathbb{E}\left[\exp\left({\sum\limits_{i=j+1}^{n}C(h+G_i)}\right)T_{1j}\right]\le C\left(\mathbb{E}\left[\exp\left({\sum\limits_{i=j+1}^{n}CG_i}\right)\right]\right)^{\frac{1}{2}}\left(\mathbb{E}\left[T_{1j}^2\right]\right)^{\frac{1}{2}}\le Ch^{2p+1},
\end{align*}
whence
\begin{equation}\label{T1}
\mathbb{E}\left[\sum_{j=0}^{n}\exp\left(\sum_{i=j+1}^{n}C(h+G_i)\right)T_{1j}\right]\le Ch^{2p}.
\end{equation}
Now we estimate the second summand in \eqref{En+1}.
By the definition of $T_{2j}$ and using the Malliavin integration by parts formula (see e.g. \cite[Lemma 1.2.1]{DN06}), we obtain
\begin{align*}
&\mathbb{E}\left[\left(\prod_{i=j+1}^{n}(1+C(h+G_i))\right)T_{2j}\right]\\
&=C\mathbb{E}\left[\left(\prod_{i=j+1}^{n}(1+C(h+G_i))\right)\|Q_j-Q(t_j)\|^{2p-2} (Q(t_j)-Q_j)^\top\left(\sum_{k=1}^d \int_{t_j}^{t_{j+1}}\int_{t_j}^t \sigma_k\,\ud W_{s}^{k}\,\ud t\right)\right]\\
&=C\sum_{k=1}^d\int_{t_j}^{t_{j+1}}\mathbb{E}\left[\left(\prod_{i=j+1}^{n}(1+C(h+G_i))\right)\|Q_j-Q(t_j)\|^{2p-2}(Q(t_j)-Q_j)^\top\sigma_k\int_{t_j}^t\,\ud W_{s}^{k}\right]\ud t\\
&=C\sum_{k=1}^d\int_{t_j}^{t_{j+1}}\mathbb{E} \Bigg[\int_{t_j}^t D_r^k\Bigg[\left(\prod_{i=j+1}^{n}\left(1+C(h+G_i)\right)\right)\|Q_j-Q(t_j)\|^{2p-2}(Q(t_j)-Q_j)^\top\sigma_k\Bigg]\,\ud r\Bigg]\ud t.
\end{align*}
The chain rule leads to
\begin{align*}
&\mathbb{E} \left[\int_{t_j}^t D_r^k\left[\left(\prod_{i=j+1}^{n}(1+C(h+G_i))\right)\|Q_j-Q(t_j)\|^{2p-2}(Q(t_j)-Q_j)^\top\sigma_k\right]\,\ud r\right]\\
&=\mathbb{E} \left[\int_{t_j}^t D_r^k\left(\prod_{i=j+1}^{n}\left(1+C(h+G_i)\right)\right)\|Q_j-Q(t_j)\|^{2p-2}(Q(t_j)-Q_j)^\top\sigma_k\,\ud r\right]\\
&\quad+\mathbb{E} \left[\int_{t_j}^t \left(\prod_{i=j+1}^{n}\left(1+C(h+G_i)\right)\right)D_r^k\left(\|Q_j-Q(t_j)\|^{2p-2}\left(Q(t_j)-Q_j\right)^\top\sigma_k\right)\,\ud r\right]\\
&=\mathbb{E} \left[\int_{t_j}^t D_r^k\left(\prod_{i=j+1}^{n}\left(1+C(h+G_i)\right)\right)\|Q_j-Q(t_j)\|^{2p-2}\left(Q(t_j)-Q_j\right)^\top\sigma_k\,\ud r\right],
\end{align*}
where we used the fact that $D_r^k\left(\|Q_j-Q(t_j)\|^{2p-2}\left(Q(t_j)-Q_j\right)^\top\sigma_k\right)$ is zero almost everywhere in $(t_j,t]\times\Omega$ since $Q_j-Q(t_j)$ is $\mathscr{F}_{t_j}$-measurable (see e.g. \cite[Corollary 1.2.1]{DN06}).  
Then the H\"{o}lder inequality, Lemma \ref{DH} and the Young inequality yield that
\begin{align*}
&\mathbb{E}\left[\sum_{j=0}^{n}\left(\prod_{i=j+1}^{n}\left(1+C(h+G_i)\right)\right)T_{2j}\right]\\
&\le C\sum_{j=0}^{n}\sum_{k=1}^d\int_{t_j}^{t_{j+1}}\int_{t_j}^t \mathbb{E}\Bigg[D_r^k\left(\prod_{i=j+1}^{n}\left(1+C(h+G_i)\right)\right)\|Q_j-Q(t_j)\|^{2p-2}\left(Q(t_j)-Q_j\right)^\top\sigma_k\,\ud r\Bigg]\,\ud t\\
&\le C\sum_{j=0}^{n}\sum_{k=1}^d\int_{t_j}^{t_{j+1}}\int_{t_j}^t \left(\mathbb{E} \left[D_r^k\left(\prod_{i=j+1}^{n}\left(1+C(h+G_i)\right)\right)\right]^{2p}\right)^{\frac{1}{2p}}\left(\mathbb{E} \left[\|Q_j-Q(t_j)\|^{2p}\right]\right)^{\frac{2p-1}{2p}}\,\ud r\,\ud t\\
&\le \sum_{j=0}^{n}Ch^2\left(\mathbb{E} \left[\|Q_j-Q(t_j)\|^{2p}\right]\right)^{\frac{2p-1}{2p}}= \sum_{j=0}^{n}Ch^{\frac{2p+1}{2p}}\left(h\mathbb{E} \left[\|Q_j-Q(t_j)\|^{2p}\right]\right)^{\frac{2p-1}{2p}}\\
&\le\sum_{j=0}^{n}h\mathbb{E} \left[\mathcal{S}_j^{2p}\right]+\sum_{j=0}^{n}h^{2p+1}.
\end{align*}
Combining \eqref{En+1}, \eqref{T1} and the discrete Gronwall lemma, we complete the proof.
\qed

Similar to \cite[Corollary 4.1]{BCH18}, from the Theorem \ref{SC1} above, we conclude the following stronger error estimation immediately. 
\begin{cor}
Let Assumption \ref{F2} hold, $h_0$ be a sufficiently small positive constant and $p\ge1$. Then for arbitrary $0<\delta<1$, there exists some positive constant $C=C(p,T,\sigma,\delta,X(0))$ such that for any $h\in(0,h_0]$,
\begin{equation*}
\left\|\sup_{n\le N^h}\|X_n-X(t_n)\|\right\|_{L^{2p}(\Omega)}\le Ch^{\delta}.
\end{equation*}
\end{cor}
\begin{proof}
Owing to Theorem \ref{SC1}, we deduce that
\begin{align*}
\mathbb{E}\left[\left\|\sup_{n\le N^h}\|X_n-X(t_n)\|\right\|^{2q}\right]\le\mathbb{E}\left[\sum_{n=1}^{N^h}\left\|X_n-X(t_n)\right\|^{2q}\right]\le Ch^{2q-1},\,\forall\,q\ge 1.
\end{align*}
By choosing $1-\frac{1}{q}\ge \delta$ and $q\ge p$, we finish the proof.
\end{proof}

\section{Convergence in probability density function}\label{S6}

In Sections \ref{S3} and \ref{S4}, we have shown the existence of density functions of  $X(t)$, $t\in(0,T]$ and  $X_n$, $n=2,\cdots,N^h$. It is natural to ask what the relationship between these density functions is.
In this section,  we show that the density function of 
$X(T)$ can be approximated by that of $X_{N^h}$. 
Meanwhile, the approximation error between the density functions is analyzed.

\subsection{Convergence in $\mathbb{D}^{\alpha,p}(\mathbb{R}^{2m})$}

We consider the convergence in $\mathbb{D}^{\alpha,p}(\mathbb{R}^{2m})$ in this part, which is a nature extension of  the convergence in $L^{2p}(\Omega;\mathbb{R}^{2m})$ of the proposed scheme \eqref{split sol}. We also remark that convergence in $\mathbb{D}^{1,p}$ for It\^{o}-Taylor approximation solution for general SDEs whose coefficients are smooth with bounded derivatives has been shown in \cite{HW96}.

\begin{tho}\label{MDC} 
Let Assumption \ref{F2} hold, $h_0$ be a sufficiently small positive constant and $\alpha,\,p\ge1$ be two integers. There exists some positive constant $C=C(p,T,\sigma,\alpha,X(0))$ such that for any $h\in(0,h_0]$,
\begin{equation}\label{Mdc}
\sup_{n\le N^h}\left\|D^\alpha X_n-D^\alpha X(t_n)\right\|_{L^p(\Omega;\mathbb{H}^{\otimes\alpha}\bigotimes\mathbb{R}^{2m})}\le Ch.
\end{equation}
\end{tho}
\begin{proof}
We prove \eqref{Mdc} by induction on $\alpha$.
For $\alpha=1$, by the H\"{o}lder inequality, there exists $C>0$ such that 
\begin{align*}
\|DX_n-DX(t_n)\|_{L^p(\Omega;\mathbb{H}\bigotimes\mathbb{R}^{2m})}^p
&=\mathbb{E}\left|\int_0^T\|D_{r_1}X_n-D_{r_1}X(t_n)\|^2\,\ud r\right|^{\frac{p}{2}}\\
&\le C\int_0^T\mathbb{E}\|D_{r_1}X_n-D_{r_1}X(t_n)\|^p\,\ud r.
\end{align*}
Thus, it suffices to show that for any fixed $r_1\in(0,T]$, 
\begin{equation*}
\sup_{n\le N^h}\mathbb{E}\|D_{r_1}X_n-D_{r_1}X(t_n)\|^{p}\le Ch^{p}.
\end{equation*}
Let $r_1\in(t_i,t_{i+1}]$ for some integer $0\le i\le N^h-1$. Taking the Malliavin derivatives on both sides of \eqref{P} and \eqref{Q} respectively, then for $i<n\le N^h-1$, 
\begin{align}\label{DDP}
D_{r_1}P_{n+1}-D_{r_1}P(t_{n+1})&=e^{-vh}(D_{r_1}P_n-D_{r_1}P(t_n))\\\nonumber
&\quad+e^{-vh}\int_{t_n}^{t_{n+1}}\int_0^1\int_0^1 \nabla^2 F\left(\theta Q(t)+(1-\theta)\left(Q_n+\tau(Q_{n+1}-Q_n)\right)\right)\\\nonumber
&\quad(D_{r_1}Q(t_n)-D_{r_1}Q_n)\,\ud \theta\,\ud \tau\,\ud t+S_{1n},\\\label{DDQ}
D_{r_1}Q_{n+1}-D_{r_1}Q(t_{n+1})&=D_{r_1}Q_n-D_{r_1}Q(t_n)+h(D_{r_1}P_n-D_{r_1}P(t_n))+S_{2n},
\end{align}
where
\begin{align*}
S_{1n}&=S^{11}_n+S^{12}_n+S^{13}_n+S^{14}_n,\\
S_{2n}&=S^{21}_n+S^{22}_n+S^{23}_n,\\
S^{11}_n&=\int_{t_n}^{t_{n+1}}\left[-e^{-vh}+e^{-v(t_{n+1}-t)}\right]\nabla^2 F(Q(t))D_{r_1}Q(t)\,\ud t,\\
S^{12}_n&=e^{-vh}\int_{t_n}^{t_{n+1}}\int_0^1\int_0^1 D_{r_1}\left[\nabla^2 F\left(\theta Q(t)+\left(1-\theta\right)\left(Q_n+\tau\left( Q_{n+1}-Q_n\right)\right)\right)\right]\\
&\quad(Q(t_n)-Q_n)\,\ud \theta\,\ud \tau\,\ud t,\\
S^{13}_n&=e^{-vh}\int_{t_n}^{t_{n+1}}\int_0^1\int_0^1 D_{r_1}\left[\nabla^2 F\left(\theta Q(t)+\left(1-\theta\right)\left(Q_n+\tau\left(Q_{n+1}-Q_n\right)\right)\right)\right]\\
&\quad\left(\int_{t_n}^{t}  P(s)\,\ud s-\frac{\tau h}{2}(P_n+\bar P_{n+1})\right)\,\ud \tau\,\ud \theta\,\ud t,\\
S^{14}_n&=e^{-vh}\int_{t_n}^{t_{n+1}}\int_0^1\int_0^1 \nabla^2 F\left(\theta Q(t)+\left(1-\theta\right)\left(Q_n+\tau\left(Q_{n+1}-Q_n\right)\right)\right)\\
&\quad D_{r_1}\left[\int_{t_n}^{t}  P(s)\,\ud s-\frac{\tau h}{2}(P_n+\bar P_{n+1})\right]\,\ud \tau\,\ud \theta\,\ud t,\\
S^{21}_n&=\int_{t_n}^{t_{n+1}}\int_{t_n}^{t}e^{-v(t-s)}\nabla^2 F(Q(s))D_{r_1}Q(s)\,\ud s\,\ud t,\\
S^{22}_n&=\left(h-\frac{1-e^{-vh}}{v}\right)D_{r_1}P(t_n),\\
S^{23}_n&=-\frac{h^2}{2}\int_0^1D_{r_1}\left[\nabla F(Q_n+\tau(Q_{n+1}-Q_n))\right]\,\ud\tau.
\end{align*}
Applying the triangle inequality yields
\begin{align*}
\|D_{r_1}P_{n+1}-D_{r_1}P(t_{n+1})\|&\le \|D_{r_1}P_n-D_{r_1}P(t_n)\|+G_n\|D_{r_1}Q(t_n)-D_{r_1}Q_n\|+\|S_{1n}\|,\\
\|D_{r_1}Q_{n+1}-D_{r_1}Q(t_{n+1})\|&\le \|D_{r_1}Q_n-D_{r_1}Q(t_n)\|+h\|D_{r_1}P(t_n)-D_{r_1}P_n\|+\|S_{2n}\|.
\end{align*} 
Define $\mathcal{R}_{n+1}:=\|D_{r_1}P_{n+1}-D_{r_1}P(t_{n+1})\|+\|D_{r_1}Q_{n+1}-D_{r_1}Q(t_{n+1})\|$, then 
\begin{align}\label{DE1}
\mathcal{R}_{n+1}\le \mathcal{R}_{n}+(h+G_n)\mathcal {R}_{n}+S_n,
\end{align}
where $S_n=\|S_{1n}\|+\|S_{2n}\|$. Using the H\"{o}lder inequality, the estimate \eqref{DK}, Lemmas \ref{MB}, \ref{EE} and \ref{NDI}, we obtain that for $\kappa=1,\,\iota=1,2,3,4,$
\begin{align}\label{Snp}
&\|S^{\kappa\iota}_n\|_{L^q(\Omega)}\le Ch^2,\,q \ge1, \,n=0,\ldots,N^h-1.
\end{align}
And for $\kappa=2,\,\iota=1,2,3$, \eqref{Snp} also holds.
Therefore
\begin{align}\label{Snp1}
\mathbb{E}\left[\left(\sum_{j=i+1}^{n} S_j\right)^q\right]\le(n-i)^{q-1}\sum_{j=i+1}^{n}\mathbb{E}\left[S_j^q\right]
\le Ch^{q}, \,\forall\,q\ge1.
\end{align}
For $n=i$, since $r\in (t_i,t_{i+1}]$, we get $D_{r_1}X(t_i)=0,\, D_{r_1}X_i=0$. Hence
\begin{align*}
&D_{r_1}P_{i+1}-D_{r_1}P(t_{i+1})=S_{1i},\qquad D_{r_1}Q_{i+1}-D_{r_1}Q(t_{i+1})=S^{21}_i+\frac{1-e^{-v(t_{i+1}-r)}}{v}\sigma.
\end{align*}
Combining \eqref{Snp} and the fact that $t_{i+1}-r<h$, we obtain
\begin{align}\label{Enp}
\mathbb{E}\left[\mathcal{R}_{i+1}^q\right]\le Ch^q,\,\forall\,q\ge1.
\end{align}
It follows from the discrete Gronwall lemma and \eqref{DE1} that
 for any $n=0,\ldots,N^h-1$,
\begin{align}\label{DE2}
\mathcal{R}_{n+1}^p \le C\left(\sum_{j=i+1}^{n} S_j\right)^p\exp\left(\sum_{j=i+1}^{n} p(h+G_j)\right)+C\exp\left(\sum_{j=i+1}^{n} p(h+G_j)\right)\mathcal{R}_{i+1}^p.
\end{align}
Then using estimates \eqref{SE3}, \eqref{Snp1}, \eqref{Enp} and the H\"{o}lder inequality, we complete the proof of the assertion for $\alpha=1$.

For $\alpha\ge2$, let $r_k\in(t_{i_k},t_{{i_k}+1}]$ for $0\le i_k\le N^h-1,\, k=1,\ldots,\alpha$. Taking the $\alpha$th Malliavin derivatives on both sides of \eqref{P} and \eqref{Q}, and using the chain rule, we have that  for $\max\limits_ki_k<n\le N^h-1$,
\begin{align*}
&D_{r_1,\ldots,r_\alpha}^{j_1,\ldots,j_\alpha}P_{n+1}-D_{r_1,\ldots,r_\alpha}^{j_1,\ldots,j_\alpha}P(t_{n+1})\\
&=e^{-vh}\left(D_{r_1,\ldots,r_\alpha}^{j_1,\ldots,j_\alpha}P_n-D_{r_1,\ldots,r_\alpha}^{j_1,\ldots,j_\alpha}P(t_n)\right)\\
&\quad+e^{-vh}\int_{t_n}^{t_{n+1}}\int_0^1\int_0^1 \nabla^2 F(\theta Q(t)+(1-\theta)(Q_n+\tau(Q_{n+1}-Q_n)))\\
&\qquad\qquad\qquad\qquad\qquad\quad(D_{r_1,\ldots,r_\alpha}^{j_1,\ldots,j_\alpha}Q(t_n)-D_{r_1,\ldots,r_\alpha}^{j_1,\ldots,j_\alpha}Q_n)\,\ud \theta\,\ud \tau\,\ud t
+S_{n\alpha}^1,\\
&D_{r_1,\ldots,r_\alpha}^{j_1,\ldots,j_\alpha}Q_{n+1}-D_{r_1,\ldots,r_\alpha}^{j_1,\ldots,j_\alpha}Q(t_{n+1})\\
&=D_{r_1,\ldots,r_\alpha}^{j_1,\ldots,j_\alpha}Q_n-D_{r_1,\ldots,r_\alpha}^{j_1,\ldots,j_\alpha}Q(t_n)+h\left(D_{r_1,\ldots,r_\alpha}^{j_1,\ldots,j_\alpha}P_n-D_{r_1,\ldots,r_\alpha}^{j_1,\ldots,j_\alpha}P(t_n)\right)+S_{n\alpha}^2,
\end{align*}
where
\begin{align*}
S_{n\alpha}^1&=S_{n\alpha}^{11}+S_{n\alpha}^{12}+S_{n\alpha}^{13}+S_{n\alpha}^{14},\\
S_{n\alpha}^2&=S_{n\alpha}^{21}+S_{n\alpha}^{22}+S_{n\alpha}^{23},\\
S_{n\alpha}^{11}&=\int_{t_n}^{t_{n+1}}\left[-e^{-vh}+e^{-v(t_{n+1}-t)}\right]D_{r_1,\ldots,r_\alpha}^{j_1,\ldots,j_\alpha}\left[\nabla F(Q(t))\right]\,\ud t,\\
S_{n\alpha}^{12}&=e^{-vh}\int_{t_n}^{t_{n+1}}\int_0^1\int_0^1 D_{r_1,\ldots,r_\alpha}^{j_1,\ldots,j_\alpha}\bigg[\nabla^2 F\left(\theta Q(t)+\left(1-\theta\right)\left(Q_n+\tau\left( Q_{n+1}-Q_n\right)\right)\right)\\
&\quad(Q(t_n)-Q_n)\bigg]\,\ud \theta\,\ud \tau\,\ud t-e^{-vh}\int_{t_n}^{t_{n+1}}\int_0^1\int_0^1\\ 
&\nabla^2F(\theta Q(t)+(1-\theta)(Q_n+\tau(Q_{n+1}-Q_n)))(D_{r_1,\ldots,r_\alpha}^{j_1,\ldots,j_\alpha}Q(t_n)-D_{r_1,\ldots,r_\alpha}^{j_1,\ldots,j_\alpha}Q_n)\,\ud \theta\,\ud \tau\,\ud t,\\
S_{n\alpha}^{13}&=e^{-vh}\int_{t_n}^{t_{n+1}}\int_0^1\int_0^1 D_{r_1,\ldots,r_\alpha}^{j_1,\ldots,j_\alpha}\bigg[\nabla^2 F\left(\theta Q(t)+\left(1-\theta\right)\left(Q_n+\tau\left(Q_{n+1}-Q_n\right)\right)\right)\\
&\quad\left(\int_{t_n}^{t}  P(s)\,\ud s-\frac{\tau h}{2}(P_n+\bar P_{n+1})\right)\bigg]\,\ud \tau\,\ud \theta\,\ud t,\\
S_{n\alpha}^{21}&=\int_{t_n}^{t_{n+1}}\int_{t_n}^{t}e^{-v(t-s)}D_{r_1,\ldots,r_\alpha}^{j_1,\ldots,j_\alpha}[\nabla F(Q(s))]\,\ud s\,\ud t,\\
S_{n\alpha}^{22}&=\left(h-\frac{1-e^{-vh}}{v}\right)D_{r_1,\ldots,r_\alpha}^{j_1,\ldots,j_\alpha}P(t_n),\\
S_{n\alpha}^{23}&=-\frac{h^2}{2}\int_0^1D_{r_1,\ldots,r_\alpha}^{j_1,\ldots,j_\alpha}[\nabla F(Q_n+\tau(Q_{n+1}-Q_n))]\,\ud\tau.
\end{align*}
In view of the Wiener-It\^o chaos expansion of the Malliavin derivative (see e.g. \cite[Proposition 1.2.7]{DN06}), we have $D_{r_1,\ldots,r_\alpha}^{j_1,\ldots,j_\alpha}Q_{n+1}=D_{r_{\sigma_1},\ldots,r_{\sigma_\alpha}}^{j_{\sigma_1},\ldots,j_{\sigma_\alpha}}Q_{n+1}$ for all permutations of $(1,2,\ldots,\alpha)$. Thus, for $\max\limits_ki_k=n$, without loss of generality, we assume that $n=i_1$. Then it follows that
\begin{align*}
D_{r_1,\ldots,r_\alpha}^{j_1,\ldots,j_\alpha}P_{n+1}-D_{r_1,\ldots,r_\alpha}^{j_1,\ldots,j_\alpha}P(t_{n+1})=S_{n\alpha}^1;\qquad
D_{r_1,\ldots,r_\alpha}^{j_1,\ldots,j_\alpha}Q_{n+1}-D_{r_1,\ldots,r_\alpha}^{j_1,\ldots,j_\alpha}Q(t_{n+1})=S_{n\alpha}^{21}.
\end{align*}
Subsequent proof is similar to the case of $\alpha=1$ and is omitted.
\end{proof}
\begin{rem}\label{Fk1}
In Theorem \ref{MDC}, if the condition $F\in C_p^\infty$ is replaced by  $F\in C_p^k$ for some fixed constant $k\ge2$, then by Remark \ref{Fk0}, the conclusion \eqref{Mdc} holds for any $\alpha\le k-2$ and $p\ge1$.
\end{rem}
\subsection{Convergence in probability density function}
As is well known, the first probabilistic proof of H\"ormander's theorem was given by Malliavin, whose key step is to prove that, under H\"ormander's  condition, the Malliavin covariance matrix of the exact solution of the SDE is non-degenerate. For our discrete case, in the light of Lemma \ref{NDI}, the smoothness of the density function of numerical solution $X_{N^h}$ boils down to the question of the boundedness of the moments of $\det(\gamma_{N^h})^{-1}$ as well. 

In this part, we show that the proposed numerical solution $X_{N^h}$ is uniformly non-degenerate with respect to sufficiently small stepsize $h>0$, and therefore admits a smooth density function.

\begin{tho}\label{UD}
Let Assumptions \ref{F2}-\ref{F4} hold. Then for any $1\le p<\infty$, there exists a positive constant $\nu(p)$ such that 
\begin{equation*}
\left\|\det(\gamma_{N^h})^{-1}\right\|_{L^p(\Omega)}=\mathcal{O}\left(h^{-\nu(p)}\right),~as~h\rightarrow0.
\end{equation*}
\end{tho}

\begin{proof}
Since
\begin{align}\label{det}
\det(\gamma_{N^h})^{-1}=\prod_{i=1}^{2m}\lambda_i(\gamma_{N^h})^{-1}\le\left(\lambda_{min}(\gamma_{N^h})\right)^{-2m},
\end{align}
it suffices to estimate the smallest eigenvalue of $\gamma_{N^h}$. It follows from \eqref{MX} that
\begin{align*}
\gamma_{N^h}&=A_{N^h-1}\gamma_{N^h-1}A_{N^h-1}^\top+\gamma_1\\
&=\frac{1-e^{-2vh}}{2v}\left\{\sum_{k=0}^{N^h-2}A_{N^h-1}\cdots A_{k+1}\left[\begin{array}{cc}\sigma\sigma^\top& 0 \\0 & 0 \end{array}\right]A_{k+1}^\top\cdots A_{N^h-1}^\top+\left[\begin{array}{cc}\sigma\sigma^\top& 0 \\0 & 0 \end{array}\right]\right\}.
\end{align*}
The definition of $A_{N^h-1}$ yields that
\begin{align*}
&\left[\begin{array}{c}y_1^\top,y_2^\top\end{array}\right]A_{N^h-1}\left[\begin{array}{cc}\sigma\sigma^\top& 0 \\0 & 0 \end{array}\right]A_{N^h-1}^\top\left[\begin{array}{c}y_1\\y_2\end{array}\right]\\
=&\Bigg\|e^{-vh}y_1^\top\left(1-\frac{h^2}{2}F_1(Q_{N^h-1},Q_{N^h})\right)\left(1+\frac{h^2}{2}F_1(Q_{N^h-1},Q_{N^h})\right)^{-1}\sigma\\
&\quad+hy_2^\top\left(1+\frac{h^2}{2}F_1(Q_{N^h-1},Q_{N^h})\right)^{-1}\sigma\Bigg\|^2.
\end{align*}
To simplify the notations, we introduce
\begin{align*}
&B_{N^h}:=h\left(I+\frac{h^2}{2}F_1(Q_{N^h-1},Q_{N^h})\right)^{-1},\\
&U_{N^h}:=e^{-vh}\left(I-\frac{h^2}{2}F_1(Q_{N^h-1},Q_{N^h})\right)\left(I+\frac{h^2}{2}F_1(Q_{N^h-1},Q_{N^h})\right)^{-1}.
\end{align*}
Combining the above equalities together, we get
\begin{align*}
&\lambda_{min}(\gamma_{N^h})=\min_{\substack{y=(y_1^\top,y_2^\top)^\top\in\mathbb{R}^{2m}\\\|y\|_2=1}}y^\top\gamma_{N^h}y\\
&\ge\min_{\substack{y=(y_1^\top,y_2^\top)^\top\in\mathbb{R}^{2m}\\\|y\|=1}}\frac{1-e^{-2vh}}{2v}\left[\begin{array}{c}y_1^\top,y_2^\top\end{array}\right]\left\{A_{N^h-1}\left[\begin{array}{cc}\sigma\sigma^\top& 0 \\0 & 0 \end{array}\right]A_{N^h-1}^\top+\left[\begin{array}{cc}\sigma\sigma^\top& 0 \\0 & 0 \end{array}\right]\right\}\left[\begin{array}{c}y_1\\y_2\end{array}\right]\\
&=:\frac{1-e^{-2vh}}{2v}\min_{\substack{y=(y_1^\top,y_2^\top)^\top\in\mathbb{R}^{2m}\\\|y\|=1}}f(y),
\end{align*}
where
\begin{align*}
f(y)&=\|y_1^\top U_{N^h}\sigma+y_2^\top B_{N^h}\sigma\|^2+\|y_1^\top\sigma\|^2\\
&=y_1^\top U_{N^h}\sigma\sigma^\top U_{N^h}^\top y_1+y_2^\top B_{N^h}\sigma\sigma^\top B_{N^h}^\top y_2+2y_1^\top U_{N^h}\sigma\sigma^\top B_{N^h}^\top y_2+y_1^\top\sigma\sigma^\top y_1.
\end{align*}
By splitting  the unit sphere of $\mathbb{R}^{2m}$ into $\{\|y\|=1,\,\|y_1\|\ge h^\delta\}$ and $\{\|y\|=1,\,\|y_1\|< h^\delta\}$, with $\delta>0$ being later determined, we estimate $\lambda_{min}(\gamma_{N^h})$ as
\begin{equation}\label{gam}
\lambda_{min}(\gamma_{N^h})=\frac{1-e^{-2vh}}{2v}\min_{\substack{y=(y_1^\top,y_2^\top)^\top\in\mathbb{R}^{2m}\\\|y\|=1}}\left\{\min_{\|y_1\|\ge h^\delta}f(y), \min_{\|y_1\|<h^\delta} f(y)\right\}.
\end{equation}

Next we estimate the lower bound of $f$. The estimation of $\min_{\|y_1\|\ge h^\delta}f(y)$ is trivial, since
\begin{equation}\label{gam1}
\min_{\|y_1\|\ge h^\delta}f(y)\ge\min_{\|y_1\|\ge h^\delta}y_1^\top\sigma\sigma^\top y_1\ge\lambda_{min}\left(\sigma\sigma^\top\right)h^{2\delta}.
\end{equation}
Now we turn to giving the lower bound of the term $\min_{\|y_1\|<h^\delta}f(y)$.
Let $h\le1$. The conditions $\|y_1\|^2<h^{2\delta}$ and $\|y_1\|^2+\|y_2\|^2=1$ imply that $\|y_2\|^2>1-h^{2\delta}$. 
The Young inequality gives
\begin{equation*}
2y_1^\top U_{N^h}\sigma\sigma^\top B_{N^h}^\top y_2\ge-\epsilon y_2^\top B_{N^h}\sigma\sigma^\top B_{N^h}^\top y_2-\frac{1}{\epsilon}y_1^\top U_{N^h}\sigma\sigma^\top U_{N^h}^\top y_1,\,\forall\,\epsilon>0,
\end{equation*}
which implies
\begin{equation}\label{st0}
f(y)\ge(1-\epsilon)y_2^\top B_{N^h}\sigma\sigma^\top B_{N^h}^\top y_2+\left(1-\frac{1}{\epsilon}\right)y_1^\top U_{N^h}\sigma\sigma^\top U_{N^h}^\top y_1+y_1^\top\sigma\sigma^\top y_1.
\end{equation}
Since $\lambda_{min}({F_1(Q_{N^h-1},Q_{N^h})})\ge-\frac{K}{2}$, we have
$\lambda_i\left(U_{N^h}U_{N^h}^\top\right)\le e^{-2vh}\left(\frac{1+\frac{h^2}{4}K}{1-\frac{h^2}{4}K}\right)^2,\,i=1,\ldots,m.$
Then it follows that
\begin{align*}
\left(\frac{1}{\epsilon}-1\right)y_1^\top U_{N^h}\sigma\sigma^\top U_{N^h}^\top y_1&\le\left(\frac{1}{\epsilon}-1\right)\lambda_{max}\left(\sigma\sigma^\top\right)\|U_{N^h}^\top y_1\|^2\\&\le\left(\frac{1}{\epsilon}-1\right)\lambda_{max}\left(\sigma\sigma^\top\right)e^{-2vh}\left(\frac{1+\frac{h^2}{4}K}{1-\frac{h^2}{4}K}\right)^2\|y_1\|^2.
\end{align*}
For simplicity, set $a:=e^{-2vh}\left(\frac{1+\frac{h^2}{4}K}{1-\frac{h^2}{4}K}\right)^2$. Consequently,
\begin{equation}\label{st10}
\left(1-\frac{1}{\epsilon}\right)y_1^\top U_{N^h}\sigma\sigma^\top U_{N^h}^\top y_1+y_1^\top\sigma\sigma^\top y_1\ge\lambda_{min}\left(\sigma\sigma^\top\right)\|y_1\|^2+\left(1-\frac{1}{\epsilon}\right)\lambda_{max}\left(\sigma\sigma^\top\right)a\|y_1\|^2.
\end{equation}
Notice that if $h\rightarrow0$, then $a\rightarrow1$. Thus there exists a sufficiently small stepsize $h(v,K)\le 1$ such that for all $h\le h(v,K)$, it holds $a<2$. For any $\epsilon$ such that $\frac{2\lambda_{max}\left(\sigma\sigma^\top\right)}{2\lambda_{max}\left(\sigma\sigma^\top\right)+\lambda_{min}\left(\sigma\sigma^\top\right)}\le\epsilon<1$, we have $\frac{\lambda_{max}\left(\sigma\sigma^\top\right)a}{\lambda_{max}\left(\sigma\sigma^\top\right)a+\lambda_{min}\left(\sigma\sigma^\top\right)}<\epsilon<1$. By a straightforward calculation, we deduce that 
\begin{equation}\label{st1}
\left(1-\frac{1}{\epsilon}\right)y_1^\top U_{N^h}\sigma\sigma^\top U_{N^h}^\top y_1+y_1^\top\sigma\sigma^\top y_1\ge0.
\end{equation}
Then it suffices to give the lower bound of $(1-\epsilon)y_2^\top B_{N^h}\sigma\sigma^\top B_{N^h}^\top y_2$.
By Choosing $\epsilon=\frac{2\lambda_{max}\left(\sigma\sigma^\top\right)}{2\lambda_{max}\left(\sigma\sigma^\top\right)+\lambda_{min}\left(\sigma\sigma^\top\right)}$, $h_0=\min\{h(v,K),\sqrt{\frac{2}{K}},1\}$,  the inequality \eqref{st1}, together with \eqref{st0}, implies that 
\begin{align*}
f(y)&\ge(1-\epsilon)y_2^\top B_{N^h}\sigma\sigma^\top B_{N^h}^\top y_2\\
&\ge\frac{\lambda_{min}\left(\sigma\sigma^\top\right)}{2\lambda_{max}\left(\sigma\sigma^\top\right)+\lambda_{min}\left(\sigma\sigma^\top\right)}\lambda_{min}\left(B_{N^h}\sigma\sigma^\top B_{N^h}^\top\right)\|y_2\|^2.
\end{align*}
It remains to evaluate the minimum eigenvalue of the symmetric positive definite matrix $B_{N^h}\sigma\sigma^\top B_{N^h}^\top$. 
By utilizing \cite[Lemma 1]{YG97}, we obtain 
\begin{align*}
\lambda_{min}\left(B_{N^h}\sigma\sigma^\top B_{N^h}^\top\right)&\ge\det\left(B_{N^h}\sigma\sigma^\top B_{N^h}^\top\right)\cdot\left(\frac{m-1}{\|B_{N^h}\sigma\sigma^\top B_{N^h}^\top\|_{\mathbb F}^2}\right)^{\frac{m-1}{2}}\\
&=C(m)\det\left(B_{N^h}\sigma\sigma^\top B_{N^h}^\top\right)\cdot\frac{1}{\|B_{N^h}\sigma\sigma^\top B_{N^h}^\top\|_{\mathbb F}^{m-1}},
\end{align*}
where $\left\|B_{N^h}\sigma\sigma^\top B_{N^h}^\top\right\|_{\mathbb F}\le\|B_{N^h}\|^2_{\mathbb F}\left\|\sigma\sigma^\top\right\|_{\mathbb F}$. Here $\|\cdot\|_{\mathbb F}$ is the Frobenius norm.
 Since $B_{N^h}$ is  symmetric positive definite, we have
$\|B_{N^h}\|^2_{\mathbb F}=tr(B_{N^h}B_{N^h}^\top)=tr\left(B_{N^h}^2\right)\le m\lambda_{max}^2(B_{N^h}).$
The spectral mapping theorem and $h<\sqrt{\frac{2}{K}}$ lead to
\begin{equation*}
\lambda_{max}(B_{N^h})=h\left(1+\frac{h^2}{2}\lambda_{min}(F_1(Q_{N^h-1},Q_{N^h}))\right)^{-1}\le h\left(1-\frac{h^2}{4}K\right)^{-1}<2h.
\end{equation*} 
Notice that
\begin{equation*}
\det\left(B_{N^h}\sigma\sigma^\top B_{N^h}^\top\right)=\det(B_{N^h})^2\det\left(\sigma\sigma^\top\right)\ge\det\left(\sigma\sigma^\top\right)\lambda_{min}^{2m}(B_{N^h}).
\end{equation*} 
Combining the above inequalities together, we have
\begin{equation}\label{st3}
\lambda_{min}\left(B_{N^h}\sigma\sigma^\top B_{N^h}^\top\right)\ge C(m,\sigma)\frac{\lambda_{min}^{2m}(B_{N^h})}{\lambda_{max}^{2m-2}(B_{N^h})}\ge C(m,\sigma)h^2\left(1+\frac{h^2}{2}\lambda_{max}(F_1(Q_{N^h-1},Q_{N^h}))\right)^{-2m}.
\end{equation}
Inserting \eqref{st1} and \eqref{st3} into \eqref{st0},  we obtain that for $\|y_1\|<h^\delta$ with $h\le h_0$,
\begin{equation}\label{gam2}
f(y)\ge C(m,\sigma)h^2\left(1+\frac{h^2}{2}\lambda_{max}(F_1(Q_{N^h-1},Q_{N^h}))\right)^{-2m}\left(1-h^{2\delta}\right).
\end{equation}

Combining \eqref{gam}, \eqref{gam1} and \eqref{gam2}, we get
\begin{align*}
&\lambda_{min}(\gamma_{N^h})\ge\\
&\frac{1-e^{-2vh}}{2v}\min\left\{\lambda_{min}\left(\sigma\sigma^\top\right)h^{2\delta},\,C(m,\sigma)h^2\left(1+\frac{h^2}{2}\lambda_{max}(F_1(Q_{N^h-1},Q_{N^h}))\right)^{-2m}\left(1-h^{2\delta}\right)\right\}.
\end{align*}
Taking its reciprocal leads to
\begin{align*}
\lambda_{min}^{-1}(\gamma_{N^h})&\le\frac{2v}{1-e^{-2vh}}\max\left\{\frac{1}{\lambda_{min}\left(\sigma\sigma^\top\right)h^{2\delta}},\,\frac{\left(1+\frac{h^2}{2}\lambda_{max}(F_1(Q_{N^h-1},Q_{N^h}))\right)^{2m}}{C(m,\sigma)h^2(1-h^\delta)}\right\}.
\end{align*}
It follows from Lemma \ref{EE} that $\mathbb{E}\left|\lambda_{max}(F_1(Q_{N^h-1},Q_{N^h}))\right|^p\le C(p,T)$ holds for any $p\ge1$. Since  $\frac{2v}{1-e^{-2vh}}=\mathcal{O}(h^{-1})$ as $h\rightarrow0$,  we get for any $p\ge1$,
$$\mathbb{E}\left|\lambda_{min}^{-1}(\gamma_{N^h})\right|^p\le C(p,m,\sigma)\max\left\{h^{-2\delta p},h^{-2p}\right\}h^{-p}.$$
Taking $\delta=1$ and using \eqref{det}, the desired result $\left\|\det(\gamma_{N^h})^{-1}\right\|_{L^p(\Omega)}=\mathcal{O}\left(h^{-\nu(p)}\right)$, $\nu(p)\le 6m$ as $h\rightarrow 0$ follows. 
\end{proof}

\begin{cor}
Let Assumptions \ref{F2}-\ref{F4} hold. Then $X_{N^h}$ admits a smooth density function.
\end{cor}
\begin{proof}
It follows immediately from Lemma \ref{NDI} and Theorem \ref{UD}.
\end{proof}
\begin{rem}
If the coefficient $F\in C_p^k$ for some fixed constant $k\ge2$, then from Proposition 1.1 and \cite[Proposition 5.4]{SSM05}, 
the density functions of $X(T)$ and $X_{N^h}$ belong to $C^\alpha$ for some $\alpha=\alpha(k)$.
\end{rem}
Now we are in the position to deduce the convergence rate in density of scheme \eqref{split sol} for equation \eqref{SDE1}.

\textit{Proof of Theorem \ref{order}}\quad Let $h_0$ be a sufficiently small positive constant.
Theorem \ref{UD}, together with Lemma \ref{NS}, Theorem \ref{MDC} and Lemma \ref{pdf0} indicates that $\det(\gamma_{N^h})^{-1}$ has moments of all orders uniformly with respect to $h\in(0,h_0]$, i.e.,
\begin{equation*}
\sup_{h\in(0,h_0]}\left\|\det(\gamma_{N^h})^{-1}\right\|_{L^p(\Omega)}<\infty,
\end{equation*}
which combined with Lemma \ref{NS}, Theorem \ref{MDC}, Proposition \ref{pdf} and Remark \ref{pdf1} completes the proof.\qed

\begin{cor}
Let Assumptions \ref{F2}-\ref{F4} hold. Let $\beta\ge0, 1<q<\infty$ and $\alpha>\beta+2m/q+1$ and $G\in \mathbb{D}^{\alpha,q}, 1/p+1/q=1$. Then
\begin{align*}
\sup_{y\in\mathbb{R}^{2m}}\left|(1-\Delta)^{\beta/2}\mathbb{E}\left[G\cdot\delta_y\circ X_{N^h}\right]-(1-\Delta)^{\beta/2}\mathbb{E}\left[G\cdot\delta_y\circ X(T)\right]\right|=\mathcal{O}(h)~as~h\rightarrow0.
\end{align*}
In particular, set $G=1$, then we have
\begin{align*}
\sup_{y\in\mathbb{R}^{2m}}\left|(1-\Delta)^{\beta/2}p^{N^h}_T(X_0,y)-(1-\Delta)^{\beta/2}p_T(X(0),y)\right|=\mathcal{O}(h)~as~h\rightarrow0.
\end{align*}
where $p^{N^h}_T(X_0,y)=\mathbb{E}\left[\delta_y\circ X_{N^h}\right]$, $p_T(X(0),y)=\mathbb{E}\left[\delta_y\circ X(T)\right]$ are the density functions of $X_{N^h}$ and $X(T)$, respectively.
\end{cor}
\begin{proof}
Lemma \ref{NS} and Lemma \ref{NDI} and Theorem \ref{UD} yield
that $X(T),\,X_{N^h}$ are non-degenerate functionals. Thus, from \eqref{delta}, it follows that  for any $\alpha>\beta+2m/q+1,\,1/p+1/q=1$, we have
$(1-\Delta)^{\beta/2}\delta_y\circ X(T)\in\mathbb{D}^{-\alpha,p},$ and $
(1-\Delta)^{\beta/2}\delta_y\circ X_{N^h}\in\mathbb{D}^{-\alpha,p}.$
 \cite[Theorem 4.3]{IW84} implies that  the map $y\rightarrow\mathbb{E}\left[G\cdot\delta_y\circ X_{N^h}\right]$ is $\beta$-times continuously differentiable.
From the definition of $\mathbb{D}^{-\alpha,p}$, it follows that
\begin{align*}
&(1-\Delta)^{\beta/2}\mathbb{E}\left[G\cdot\delta_y\circ X_{N^h}\right]-(1-\Delta)^{\beta/2}\mathbb{E}\left[G\cdot\delta_y\circ X(T)\right]\\
&=\mathbb{E}\left[G\cdot\left\{(1-\Delta)^{\beta/2}\left[\delta_y\circ X_{N^h}\right]-(1-\Delta)^{\beta/2}\left[\delta_y\circ X(T)\right]\right\}\right]\\
&\le\left\|(1-\Delta)^{\beta/2}\delta_y\circ X_{N^h}-(1-\Delta)^{\beta/2}\delta_y\circ X(T)\right\|_{-\alpha,p}\|G\|_{\alpha,q}.
\end{align*}
Taking supremum over $y\in\mathbb{R}^{2m}$, we complete the proof.
\end{proof}

\section{Numerical experiments}\label{S7}
In this section, we implement some numerical tests to verify our theoretic result on the strong convergence rate of scheme \eqref{split sol}. In particular, we consider the following two stochastic Langevin equations.

\textit{Example 1:} Taking $m=1$, $d=1$ and $F(Q)=Q^4$, consider the following $2$-dimensional Langevin equation 
\begin{equation}\label{example1}
\begin{split}
&\,\ud P=-4Q^3\,\ud t-vP\,\ud t+\sigma\,\ud W(t),\\
&\,\ud Q=P\,\ud t,
\end{split}
\end{equation}
where $v>0, \,\sigma$ are fixed constants.

\textit{Example 2:} Taking $m=2$, $d=2$ and $F(Q)=Q_1^8+Q_2^2+2Q_1Q_2$, consider the following $4$-dimensional Langevin equation 
\begin{equation}\label{example2}
\begin{split}
&\,\ud P_1=-8Q_1^7\,\ud t-2Q_2\,\ud t-vP_1\,\ud t+\sigma_{11}\,\ud W_1(t)+\sigma_{12}\,\ud W_2(t),\\
&\,\ud P_2=-2Q_2\,\ud t-2Q_1\,\ud t-vP_2\, \ud t+\sigma_{21}\,\ud W_2(t)+\sigma_{22}\,\ud W_2(t),\\
&\,\ud Q_1=P_1\,\ud t,\\
&\,\ud Q_2=P_2\,\ud t,
\end{split}
\end{equation}
where $v>0,\, \sigma_{ij},\, i,\,j=1,2$ are fixed constants.

In the following experiments, we choose $\sigma=1,\,P(0)=Q(0)=1$ in equation \eqref{example1} and $\sigma_{ij}=1,\,i,j=1,2,\,P_i(0)=Q_i(0)=1,\,i=1,2$ in equation \eqref{example2}.
Errors in mean square sense of the numerical solutions against stepsize $h$ on a log-log scale are shown in Figure \ref{fig2}.
In this experiment, we compute the mean square errors at the final time $T=1$ with time steps ranging from $h=2^{-7}$ to $h=2^{-11}$,  respectively. The reference solution is computed by using the tamed Euler scheme with stepsize $h_{ref}=2^{-14}$. The expectation is realized by using the average of 200 samples and 2000 samples, which are represented by green and blue solid lines, respectively. 
The reference red dashed line has slope $1$. Figure \ref{fig2} illustrates that the strong convergence order of the splitting AVF scheme \eqref{split sol} is consistent with the theoretical result in Theorem \ref{SC1}.
\begin{figure}
	\centering
	\subfloat[$m=1,\,v=1$]{\includegraphics[width=0.5\columnwidth]{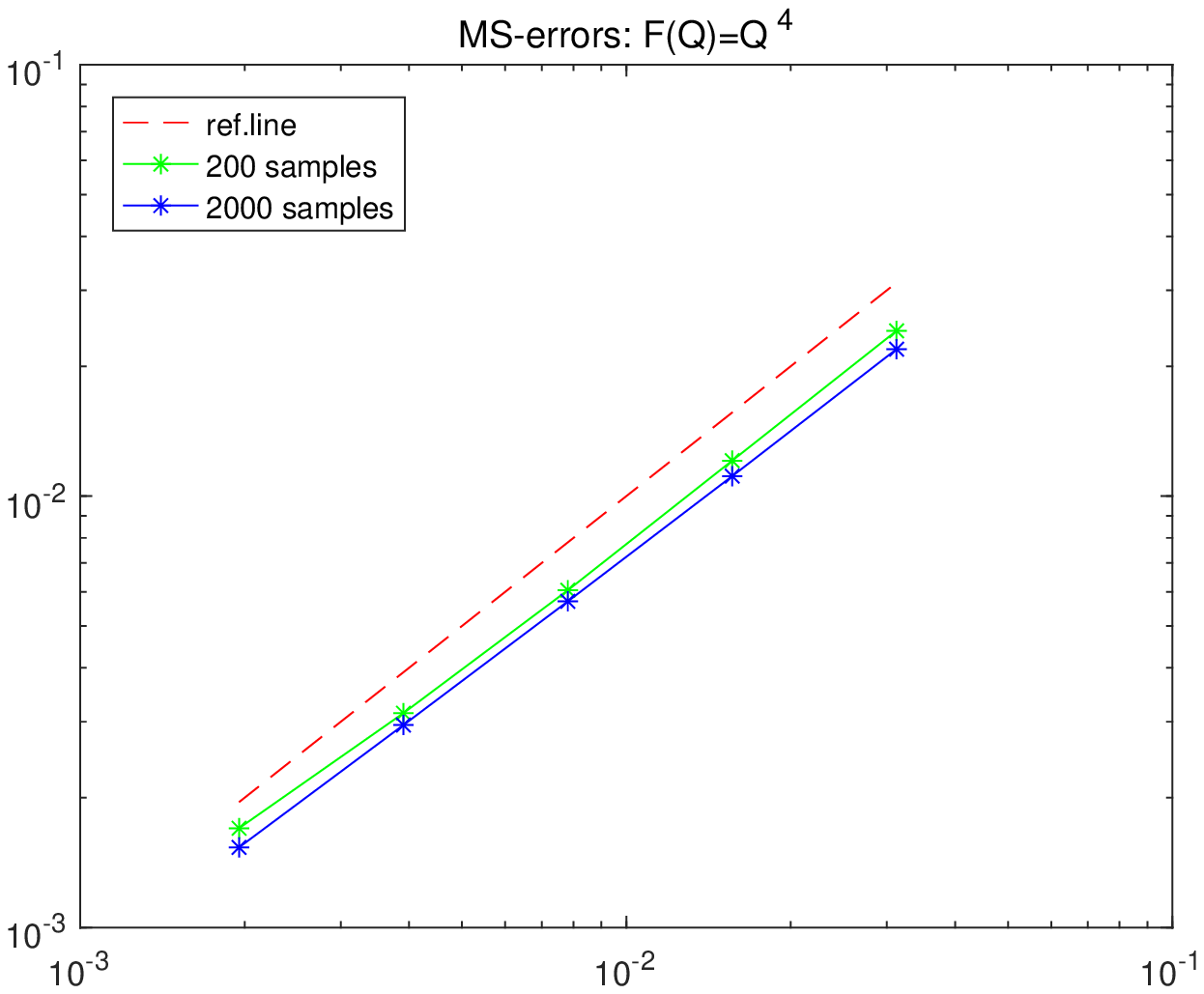}}
	\subfloat[$m=2,\,v=1$]{\includegraphics[width=0.5\columnwidth]{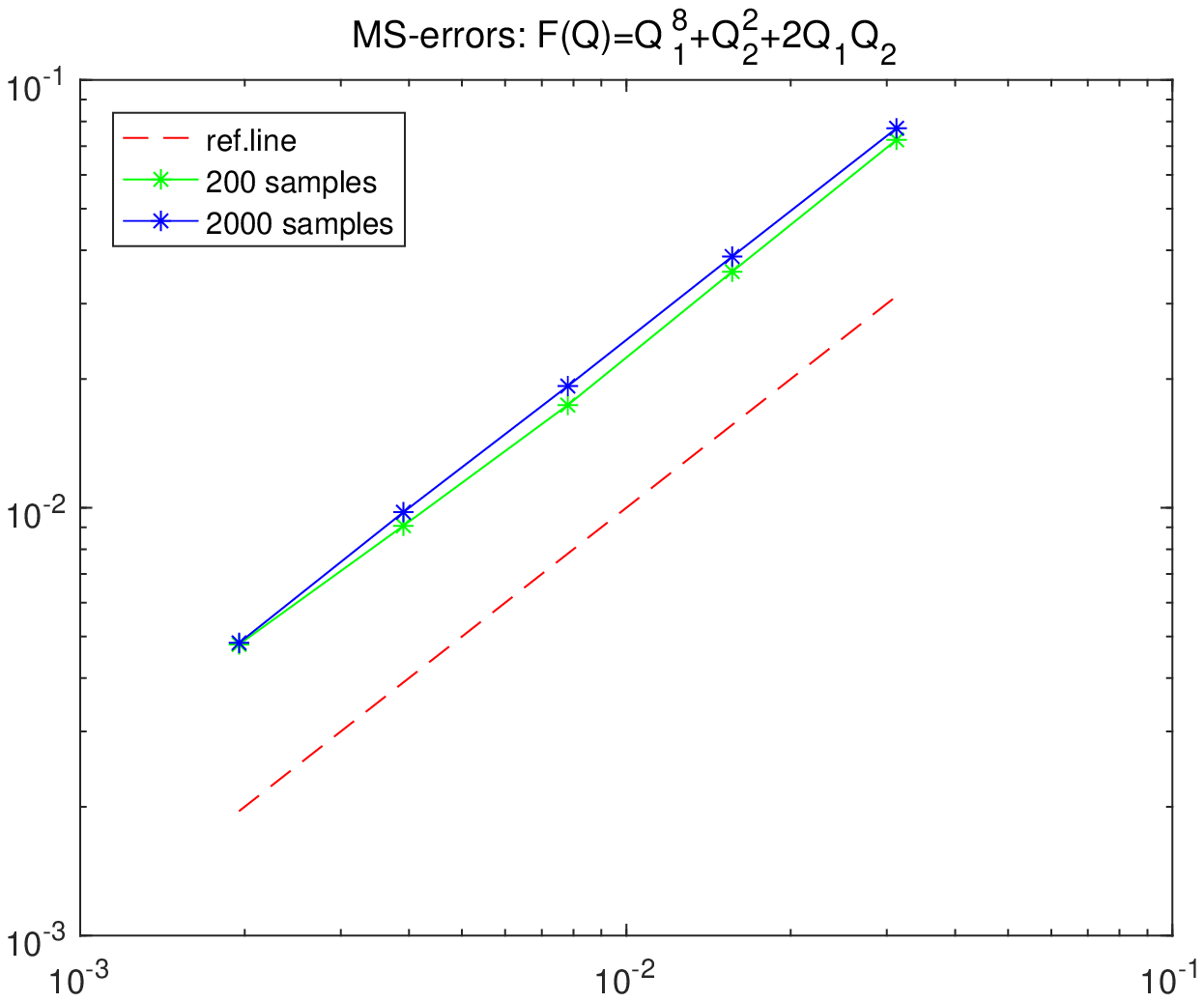}}\\
	\caption{Mean square convergence rate of splitting AVF method for stochastic Langevin equations.}\label{fig2}
	\vspace{0.2in}
\end{figure}
\\

\textbf{Acknowledgments.}
The authors are very grateful to Professor Yaozhong Hu (University of Alberta)  for his helpful discussions and suggestions.

\bibliographystyle{plain}
\bibliography{pdfreference}

\end{document}